\documentclass[11pt,reqno]{amsart}
\usepackage{amsmath, amssymb, amsthm}
\usepackage{url}
\usepackage[breaklinks]{hyperref}
\usepackage{autonum}
\usepackage{color}

\theoremstyle{plain}
\newtheorem{thm}{Theorem}[section]
\theoremstyle{plain}
\newtheorem{lem}[thm]{Lemma}

\newtheorem{cor}[thm]{Corollary}

\theoremstyle{definition}

\begin{document}

\title{A Weighted Divisor problem and Exponential Sum}

\date{}

\author{Kritika Aggarwal}
\address{Kritika Aggarwal\\ Department of Mathematics\\
Indraprastha Institute of Information Technology IIIT, Delhi\\
Okhla, Phase III, New Delhi-110020, India.} 
\email{kritikaa@iiitd.ac.in}

\author{Debika Banerjee}
\address{Debika Banerjee\\ Department of Mathematics\\
Indraprastha Institute of Information Technology IIIT, Delhi\\
Okhla, Phase III, New Delhi-110020, India.} 
\email{debika@iiitd.ac.in}


\begin{abstract}
In this paper, we investigate a weighted divisor problem involving the exponential sum of $D_{(1)}(n)$, the $n$th coefficient in the Dirichlet series expansion of $\zeta'(s)^2$. We establish a truncated Vorono\"{i} type formula for the error term of $\sum_{n\leq x}D_{(1)}(n)e(nh/k)$, analogous to the results obtained by Jutila. Utilizing this truncated formula, we derive a mean square estimate of the error term. In addition, we study the associated Riesz sum and the corresponding error term, along with its mean square estimate.
\end{abstract}

\maketitle

\section{Introduction}
For $x>0$, $\alpha\in\mathbb{R}$ and $\{b(n)\}$ be any sequence 
\begin{align}
    M(x,\alpha)=\sum_{n\leq x}b(n)e(n\alpha),\label{generalexpsum}
\end{align}
is a type of exponential sum whose upper bound estimates are often desirable in dealing with analytical problems. These estimates play an important role in the study of distribution properties of arithmetical functions and in bounding error terms in asymptotic formulas. We use the standard notation $e_k(\alpha)=e^{2\pi i\alpha/k}$ and $e(\alpha)=e^{2\pi i\alpha}$. \\

In 1937, Vinogradov investigated the exponential sum $f(x) = \sum_{p \leq x} (\log p)e(p x)$ as the generating function to refine the Hardy-Littlewood circle method, and used his technique for estimating
the exponential sums (refer \cite{VNGR} and \cite{VNGR2} for details). His work extended a 1923 result of Hardy and Littlewood, who had shown that every sufficiently large odd integer can be expressed as the sum of three primes under the assumption of the Generalized Riemann Hypothesis. Vinogradov's breakthrough provided an unconditional proof of this result and had further applications in the famous Waring's problem. Other recent works on exponential sums over primes can be found in \cite{AL} and \cite{VR}.\\

There have been significant developments in the theory of exponential sums over the last few decades. In 1987, Jutila \cite{JM2} investigated $M(x,\alpha)$ for $b(n)=\tau(n)$ to obtain an upper bound $\ll x^6$, where $\tau(n)$ is the Ramanujan $\tau$ function. Maier and Sankaranarayanan \cite{MS} in 2005 estimated the upper bound of the sum in \eqref{generalexpsum} for $b(n)=\mu(n)$, the M\"{o}bius function. This work was based on the prior estimations given by Baker and Harman \cite{BH} in 1991. Pandey \cite{MP} in 2022 provided an estimate for $\int_0^1|M(x,\alpha)|^s$ for $b(n)=\tau_k(n),$ the $k$th Dirichlet-Piltz divisor function, for $k\geq 2$, $s>2$ any real number. Most recently, the second author with Gupta \cite{DS} studied the higher power moments of exponential sum with twisted divisor function.\\

The Dirichlet divisor problem has been among the most fascinating open problems in number theory till date. Mathematicians have been attempting to refine the bound of the error term $\Delta(x)$ in the Dirichlet's asymptotic formula
\begin{align}
\sum_{n\leq x}d(n)=x\log x+(2\gamma-1)x+\Delta(x),\label{divisorproblemformula}
\end{align}
and it is conjectured that $\Delta(x)=O(x^{\frac14+\varepsilon})$ based on the mean value results. In 1987, Jutila \cite{JM} incorporated the exponential factor in \eqref{divisorproblemformula} to obtain a new kind of divisor problem 
\begin{align}
    D(x,h/k)=\sideset{}{'}\sum_{n\leq x}d(n)e_k(nh),\label{D(x,r)defn}
\end{align}
and analogously,
\begin{align}
    A(x,h/k)=\sideset{}{'}\sum_{n\leq x}a(n)e_k(nh),\label{A(x,r)defn}
\end{align}
where $a(n)$ is the Fourier coefficient of a cusp form of weight $\kappa$. He gave the truncated Vorono\"{i} formula for the error terms arising in the asymptotic formulas of \eqref{D(x,r)defn} and \eqref{A(x,r)defn}, and also studied their mean square estimates. For this, Jutila directly employed the functional equation of the allied Dirichlet series 
\begin{align}
    E(s,h/k)=\sum_{n=1}^{\infty}d(n)e(nh/k)n^{-s}\hspace{2mm},\hspace{2mm}(\Re(s)>1)\label{E(sh/k)defn}
\end{align}
and
\begin{align}
    \phi(s,h/k)=\sum_{n=1}^{\infty}a(n)e(nh/k)n^{-s}\hspace{2mm},\hspace{2mm}(\Re(s)>(\kappa+1)/2).
\end{align}

He also gave identities for the general sum functions $D_a(x,h/k)$ and $A_a(x,h/k)$ defined via Riesz means (see [\cite{JM}, eq. 1.6.1-1.6.2]), and obtained the Vorono\"{i} summation formula for the corresponding error term after studying their convergence properties [\cite{JM}, Chapter 1].\\

Inspired by the Dirichlet's classical divisor function $d(n)=\sum_{d|n}1$, Minamide \cite{MM} in 2013 introduced a "new divisor function" $D_{(k)}(n)$ defined by
\begin{align}
    D_{(k)}(n)=\sum_{d|n}(\log d)^k\left(\log \frac{n}{d}\right)^k,\label{D_k(n)formula}
\end{align}
for any integer $n\geq 1$ and $k$ any natural number. This function arises as $n$th Dirichlet coefficient of the square of the $k$th derivative of the Riemann zeta function, $\zeta^{(k)}(s)$ i.e., the $n$th coefficient in the Dirichlet series of $\left(\zeta^{(k)}(s)\right)^2$ for $\Re(s)>1$. Further, the error term $\Delta_{(k)}(x)$ in the related divisor problem $\sum_{n\leq x}D_{(k)}(n)$ is given by
\begin{align}
    \Delta_{(k)}(x)=\sum_{n\leq x}D_{(k)}(n)-xP_{2k+1}(\log x),
\end{align}
where $P_j(t)$ is a degree $j$ polynomial in $t$, explicitly given in [\cite{MM}, eq. (10)]. From this, Minamide studied the behaviour of $\Delta_{(k)}(x)$ and estimated the upper bound
\begin{align}
    \Delta_{(k)}(x)=O(x^{1/3+\varepsilon}).
\end{align}
Clearly, $D_{(0)}(n)=d(n)$. Moreover, as an analogue of \eqref{divisorproblemformula}, he also gave the truncated Vorono\"{i} formula and the mean square estimate for $\Delta_{(1)}(x)$. Other kinds of generalization of the divisor function \eqref{D_k(n)formula} can be found in \cite{FMT} and \cite{FMT2}. \\

Our object in this paper is to study an interesting modification to Jutila's divisor problem \eqref{D(x,r)defn}, where we replace $d(n)$ by the new divisor function or the weighted divisor function defined in \eqref{D_k(n)formula}. To this end, we define the Riesz mean
\begin{align}
    B_a(x,h/k)=\frac{1}{a!}\sideset{}{'}\sum_{n\leq x}D_{(1)}(n)e(nh/k)(x-n)^a,\label{Rieszdefn}
\end{align}
where $a$ is a nonnegative integer and $h/k$ is a rational number. In particular, we write
\begin{align}\label{Bxr}
    B_0(x,h/k)=B(x,h/k)=\sideset{}{'}\sum_{n\leq x}D_{(1)}(n)e(nh/k),
\end{align}
and $\Delta_a(x,h/k)$ denotes the error term in the asymptotic formula of $B_a(x,h/k)$. Our main results include a truncated Voronoï-type formula for $\Delta_0(x,h/k)=\Delta(x,h/k)$, stated in Theorem~\eqref{th1} (Section~\ref{mainresults}). We also provide general identities for $B_a(x,h/k)$ and $\Delta_a(x,h/k)$ in Theorem~\eqref{Generalizationtheorem} for $a\geq 1$. As an application, we investigate mean square estimates for these quantities, with proofs given in Sections~\ref{pfofmeansq1} and~\ref{pfofmeansq2}.


\section{Main Results}\label{mainresults}

We begin by stating the truncated Vorono\"{i} type formula  for the error term of the exponential sum $B(x,h/k)$ defined in \eqref{Bxr}. A necessary ingredient for our paper is the Dirichlet series $F(s,h/k)$ which we define as
\begin{align}\label{Fsh/k}
    F(s,h/k)=\sum_{n=1}^{\infty}\frac{D_{(1)}(n)}{n^s}\hspace{2mm},\hspace{2mm}(\Re(s)>1).
\end{align} Our proofs rely on the analytic continuation of $F(s,h/k)$ to a meromorphic function as well as its functional equation, which will be discussed in Section \ref{prelims}, followed by main proofs in the subsequent sections. 
 \par We also make use of the weighted divisor function
\begin{align}
    d_{(0,1)}(n)=\sum_{d|n}(-\log d), \hspace{2mm}(\Re(s)>1). \label{d_(0,1)defn}
\end{align}
that is the $n$th coefficient in the Dirichlet series of $\zeta(s)\zeta'(s)$. This function was previously studied by Minamide \cite{MM} in connection with the weighted divisor function given in \eqref{D_k(n)formula}. 

\begin{thm}\label{th1}
Let $\Delta(x,h/k)$ denote the error term in the asymptotic formula for the exponential sum $B(x,h/k)$. Then for $x\geq 1$, $k\leq x$, and $1\leq N\ll x$, 
\begin{align}
  \Delta(x,h/k)&=(\pi\sqrt{2})^{-1}k^{1/2}x^{1/4}\left[\hspace{1mm}\sum_{n\leq N}\frac{1}{4}d(n)e_k(-n\overline{h})n^{-3/4}\log^2x\cos\bigg(\frac{4\pi\sqrt{nx}}{k}-\frac{\pi}{4}\bigg)\right.\notag\\
   & +\sum_{n\leq N}\left(\frac{1}{2}d(n)\log n+d_{(0,1)}(n)\right)e_k(-n\overline{h})n^{-3/4}\log x\cos\bigg(\frac{4\pi\sqrt{nx}}{k}-\frac{\pi}{4}\bigg)\\
     & \left.+\sum_{n\leq N}\left(\frac{1}{4}d(n)\log^2 n+d_{(0,1)}(n)\log n+D_{(1)}(n)\right)e_k(-n\overline{h})n^{-3/4}\cos\bigg(\frac{4\pi\sqrt{nx}}{k}-\frac{\pi}{4}\bigg)\right] \notag\\
   & \hspace{9.5cm}+O(kx^{\varepsilon+\frac{1}{2}}N^{-\frac{1}{2}}),
    \label{errorformula}
\end{align}
where $d_{(0,1)}(n)$ is defined in \eqref{d_(0,1)defn}. 
\end{thm}
Taking $N=k^{2/3}x^{1/3}$ and estimating the sums on the right of \eqref{errorformula} by absolute values, we obtain the following estimate for $\Delta(x,h/k)$,
\begin{cor}
For $x\geq 1$ and $k\leq x$, we have
\begin{align}\label{deltabound}
    \Delta(x,h/k)\ll k^{2/3}x^{1/3+\varepsilon}.
\end{align}
\end{cor}

As another application of Theorem \eqref{th1}, we deduce the mean square formula for $\Delta(x,h/k)$.

\begin{thm}\label{Delta0meansqtheorem}
For $k\ll X^{\frac12-\varepsilon}$, we have
    \begin{align}
        \int_1^{X}|\Delta(x,h/k)|^2dx=\frac{k}{6\pi^2}X^{3/2}\sum_{n=1}^{\infty}&n^{-3/2}\bigg[\frac{1}{16}d^2(n)\sum_{i=0}^{4}\left(-\frac{2}{3}\right)^{i}\frac{4!}{(4-i)!}\hspace{1mm}(\log X)^{4-i}\\
    &+\frac{1}{2}d(n)f_1(n)\sum_{i=0}^{3}\left(-\frac{2}{3}\right)^{i}\frac{3!}{(3-i)!}\hspace{1mm}(\log X)^{3-i}\\
    &+\frac{1}{2}(d(n)f_2(n)+f_1^2(n))\sum_{i=0}^{2}\left(-\frac{2}{3}\right)^{i}\frac{2!}{(2-i)!}\hspace{1mm}(\log X)^{2-i}\\
    &+2f_1(n)f_2(n)\sum_{i=0}^{1}\left(-\frac{2}{3}\right)^{i}\frac{1!}{(1-i)!}\hspace{1mm}(\log X)^{1-i}+f_2^2(n)\bigg]\\
&+O(k^{3/2}X^{5/4+\varepsilon}),\label{meansquareestimate}
    \end{align}
 where the arithmetic functions $f_1(n)$ and $f_2(n)$ are given by
    \begin{align}
        f_1(n)=\frac{1}{2}d(n)\log n+d_{(0,1)}(n)\hspace{3mm} \text{and} \hspace{3mm} f_2(n)=\frac{1}{4}d(n)\log^2n+d_{(0,1)}(n)\log n+D_{(1)}(n).\label{f1f2exp}
    \end{align}
\end{thm}

We now deduce an identity for the general sum function $B_a(x,h/k)$ for $a\geq 1$ and its corresponding error term, proof of which can be found in Section \ref{pfofgeneralizationth}.
\begin{thm}\label{Generalizationtheorem}
     Let $a\geq 1$ be an integer. Then for $x>0$ and $k\ll x^{\frac12-\varepsilon}$, we have 
    \begin{multline}\label{B_aformula}
        B_a(x,h/k)=\frac{x^{1+a}}{(1+a)!k}\bigg[\frac{1}{6}\log^3 x+\left(\log k-\frac{1}{2}\sum_{n=1}^{a+1}\frac{1}{n}\right)\log^2 x\\
        +\bigg(\log k\bigg(\log k+2\gamma-2\sum_{n=1}^{a+1}1/n\bigg)+\frac{1}{2}\sum_{n=1}^{a+1}\frac{1}{n^2}+\bigg(((a+1)!)^2-\frac{1}{2}\bigg)\bigg(\sum_{n=1}^{a+1}\frac{1}{n}\bigg)^2-2\gamma_1\bigg)\log x\\
    -2\log^3 k+\log^2 k\bigg(2\gamma-\sum_{n=1}^{a+1}\frac{1}{n}\bigg)+\log k\bigg(\sum_{n=1}^{a+1}\frac{1}{n^2}+\bigg((2((a+1)!)^2-1)\sum_{n=1}^{a+1}\frac{1}{n}-2\gamma\bigg)\sum_{n=1}^{a+1}\frac{1}{n}\bigg)\\
    -\frac{1}{3}\sum_{n=1}^{a+1}\frac{1}{n^3}-\bigg(\frac{1}{2}((a+1)!)^2\bigg(\sum_{n=1}^{a+1}\frac{1}{n^2}+\bigg(\sum_{n=1}^{a+1}\frac{1}{n}\bigg)^2\bigg)-\frac{1}{3}\bigg(\sum_{n=1}^{a+1}\frac{1}{n}\bigg)^2-2\gamma_1\bigg)\sum_{n=1}^{a+1}\frac{1}{n}-4\gamma_2\bigg]\\
    +\sum_{n=0}^{a}\frac{(-1)^n}{n!(a-n)!}F(-n,h/k)x^{a-n}+\Delta_a(x,h/k),
    \end{multline}
    where 
    \begin{align}
        \Delta_a(x,h/k)=&-(k/2\pi)^{a}x^{(1+a)/2}\sum_{n=1}^{\infty}\bigg\{\frac{1}{4}d(n)\log^2 x+\left(\frac{1}{2}d(n)\log n+d_{(0,1)}(n)\right)\log x\\
        &+\left(\frac{1}{4}d(n)\log^2 n+d_{(0,1)}(n)\log n+D_{(1)}(n)\right)\bigg\}n^{-(1+a)/2}\times\\
        &\times \left\{e_k(-n\overline{h})Y_{1+a}(4\pi\sqrt{nx}/k)+(-1)^a(2/\pi)e_k(n\overline{h})K_{1+a}(4\pi\sqrt{nx}/k)\right\}\\
        &+O(k^{a+1+\varepsilon}x^{a/2+\varepsilon}).\label{delta_aformula}
    \end{align}   
    where $Y_n(z)$ and $K_n(z)$ represent the Bessel functions of order $n$ defined in \eqref{Ybessel} and \eqref{Kbessel} respectively.
\end{thm}
The following theorem contains an analogous mean square estimate for $\Delta_a(x,h/k)$, $a\geq 1$.
\begin{thm}\label{Deltaameansqformula} For $k\ll x^{\frac12-\varepsilon}$ and $a\geq 1$, we have
    \begin{align}
        \int_1^{X}|\Delta_a(x,h/k)|^2dx=&\frac{1}{2\pi (a+3/2)}(k/2\pi)^{2a+1}X^{a+3/2}\sum_{n=1}^{\infty}n^{-(a+3/2)}\\
&\times\bigg[\frac{1}{16}d^2(n)\sum_{i=0}^{4}\left(\frac{-1}{a+3/2}\right)^{i}\frac{4!}{(4-i)!}\hspace{0.5mm}\log^{4-i}X\\
&+\frac{1}{2}d(n)f_1(n)\sum_{i=0}^{3}\left(\frac{-1}{a+3/2}\right)^{i}\frac{3!}{(3-i)!}\hspace{0.5mm}\log^{3-i}X\\
&+\frac{1}{2}\left(d(n)f_2(n)+f_1^2(n)\right)\sum_{i=0}^{2}\left(\frac{-1}{a+3/2}\right)^{i}\frac{2!}{(2-i)!}\hspace{0.5mm}\log^{2-i}X\\
&+2f_1(n)f_2(n)\sum_{i=0}^{1}\left(\frac{-1}{a+3/2}\right)^{i}\frac{1!}{(1-i)!}\hspace{0.5mm}\log^{1-i}X+f_2^{2}(n)\bigg]\\
&+O(k^{2a+\frac{3}{2}}X^{a+\frac{5}{4}+\varepsilon}),\label{deltaameansqestimate}
    \end{align}
    where $f_1(n)$ and $f_2(n)$ are as defined in \eqref{f1f2exp}.
\end{thm}


\section{Preliminaries}\label{prelims}
This section contains a series of lemmas which will be used repeatedly in the proofs of the main results. Throughout this paper, we will consider 
 \begin{align}\label{chichi1}\chi(s)=2(2\pi)^{s-1}\Gamma(1-s)\sin (\pi s/2)\hspace{2mm} \text{and}\hspace{2mm}\chi_1(s)=2(2\pi)^{s-1}\Gamma(1-s)\cos (\pi s/2).\end{align}
We shall make repeated use of Stirling's formula for $\Gamma(s)$. The following version is precise enough for our purposes.

\begin{lem}\label{gammaformula}\cite[p. 38, Section 21]{HR}
Let $\delta<\pi$ be a fixed positive number. Then in any fixed strip $A_1\leq\sigma\leq A_2$ with $|arg\hspace{1.5mm}s|\leq \pi-\delta$ and $|t|\geq 1$, we have
\begin{align}\label{gammabigestimate}
\Gamma(s)=\sqrt{2\pi}\hspace{1mm}|t|^{s-1/2}exp\left(-\frac{\pi|t|}{2}-it+\frac{i\pi}{2}(\sigma-1/2)sgn(t)\right)\left(1+O(1/|t|)\right),
\end{align}
and
\begin{align}\label{modgammaestimate}
    |\Gamma(s)|=\sqrt{2\pi}\hspace{1mm}|t|^{\sigma-1/2}e^{-(\pi/2)|t|}(1+O(1/|t|)).
\end{align}
Hence, we have
\begin{align}\label{gammasmallestimate}
    \Gamma(s)=O(e^{-(\pi/2)|t|}|t|^{\sigma-1/2}).
\end{align}
\end{lem}

The following lemma shows the estimate for $\chi(s)$.

\begin{lem}\label{chiestimate}
For $s=\sigma+it$, $|t|\geq1$, and $a\leq\sigma\leq b$ ($a$ and $b$ are arbitrary fixed real numbers), we have
\begin{align}
    \chi(\sigma+it)=\bigg(\frac{|t|}{2\pi}\bigg)^{\frac{1}{2}-\sigma-it}e^{i(t\pm \frac{\pi}{4})}\bigg(1+O\bigg(\frac{1}{|t|}\bigg)\bigg),\label{chi}\\
    \chi^2(\sigma+it)=\bigg(\frac{|t|}{2\pi}\bigg)^{1-2\sigma}e^{iE(t)}\bigg(1+O\bigg(\frac{1}{|t|}\bigg)\bigg),\label{chisquare}
\end{align}
where $E(t)$ is defined as
\begin{align}
   E(t)=
   \begin{cases}
       -2t\log\frac{|t|}{2\pi}+2t+\frac{\pi}{2} & (t>0),\\
       -2t\log\frac{|t|}{2\pi}+2t-\frac{\pi}{2} & (t<0).
   \end{cases}\notag
\end{align}
Moreover assume that $|(\sigma-1)/t|<1$. For the $k$th derivative of $\chi(s)$, we have
\begin{align}
    \chi^{(k)}(\sigma+it)=\chi(\sigma+it)\bigg(-\log\frac{|t|}{2\pi}\bigg)^{k}+O\bigg(|t|^{-\frac{1}{2}-\sigma}(\log |t|)^{k-1}\bigg).\label{chikderivative}
\end{align}
\end{lem}

\begin{proof}
    For the formula \eqref{chi} see [\cite{TM}, p. 78, (4.12.3)]. One can refer [\cite{SG}, p. 133, Lemma 6] for \eqref{chikderivative}.
\end{proof}

Analogous estimates for $\chi_1(s)$ can be given as follows :

\begin{lem}\label{chi1estimate}
For $s=\sigma+it$, $|t|\geq1$, and $a\leq\sigma\leq b$ ($a$ and $b$ are arbitrary fixed real numbers), we have
\begin{align}
    \chi_1(\sigma+it)=\bigg(\frac{|t|}{2\pi}\bigg)^{\frac{1}{2}-\sigma-it}e^{i(t\mp \frac{\pi}{4})}\bigg(1+O\bigg(\frac{1}{|t|}\bigg)\bigg),\label{chi1}\\
    \chi_{1}^2(\sigma+it)=\bigg(\frac{|t|}{2\pi}\bigg)^{1-2\sigma}e^{iE(t)}\bigg(1+O\bigg(\frac{1}{|t|}\bigg)\bigg),\label{chi1square}
\end{align}
where
\begin{align}
   E(t)=
   \begin{cases}
       -2t\log\frac{|t|}{2\pi}+2t-\frac{\pi}{2} & (t>0),\\
       -2t\log\frac{|t|}{2\pi}+2t+\frac{\pi}{2} & (t<0).
   \end{cases}\notag
\end{align}
Moreover assume that $|(\sigma-2)/t|<1$. For the $k$th derivative of $\chi_1(s)$, we have
\begin{align}
    \chi_1^{(k)}(\sigma+it)=\chi_1(\sigma+it)\bigg(-\log\frac{|t|}{2\pi}\bigg)^{k}+O\bigg(|t|^{-\frac{1}{2}-\sigma}(\log |t|)^{k-1}\bigg).\label{chi1kderivative}
\end{align}
\end{lem}

\begin{proof}
    For the formula \eqref{chi1}, use equation \eqref{gammabigestimate}. 
    Formula \eqref{chi1kderivative} can be proved using induction in the same way as done for the formula \eqref{chikderivative}.
\end{proof}

The next theorem involves a set of complex integrals which can be nicely represented in terms of the Bessel functions, and will be extensively used in the proofs of the main results. The proof follows directly from the theorem of residues.

\begin{lem}\label{integraltobessel}\cite[Lemma 1.5]{JM}
Let $a$ be a nonnegative integer, $\sigma_1\geq-a/2$, $\sigma_2<-a$, $T>0$, and let $C_a$ be the contour joining the points $\sigma_1-i\infty,\sigma_1-Ti,\sigma_2-Ti,\sigma_2+Ti,\sigma_1+Ti,$ and $\sigma_1+i\infty$ by straight lines. Let $X>0$, then
\begin{align}
    I_1=\frac{1}{2\pi i}\int_{C_a}\Gamma^2(1-s)\cos{\pi s}\hspace{1mm}X^s(s(s+1)\cdots(s+a))^{-1}ds=\pi X^{(1-a)/2}Y_{a+1}(2X^{1/2}),\label{firstintegral}
\end{align}
and
\begin{align}
    I_2=\frac{1}{2\pi i}\int_{C_a}\Gamma^2(1-s)X^s(s(s+1)\cdots(s+a))^{-1}ds=2(-1)^{a+1}X^{(1-a)/2}K_{a+1}(2X^{1/2}),\label{secondintegral}
\end{align}
where $Y_n(z)$ and $K_n(z)$ represent the Bessel functions of nonnegative integral order $n$, which can be given as
\begin{align}
    Y_n(X)=(2/\pi X)^{1/2}\sin\bigg(X-\frac{1}{2}n\pi-\frac{1}{4}\pi\bigg)+O(X^{-3/2}),\label{Ybessel}
\end{align}
and
\begin{align}
    K_n(X)=(\pi/2X)^{1/2}e^{-x}(1+O(X^{-1})).\label{Kbessel}
\end{align}
\end{lem}

The proof of \eqref{Ybessel} and \eqref{Kbessel} above can be found in Watson [\cite{GW}, Sections 7.2, 7.21, 7.23].\\

The following elementary lemma on exponential integrals will be of significant use in the proofs that will follow.

\begin{lem}\label{expintegrallemma}\cite[Lemma 4.3, p. 71]{TM}
Let $F(x)$ and $G(x)$ be real functions in the interval $[a,b]$ where $G(x)$ is continuous and $F(x)$ continuously differentiable. Suppose that $G(x)/F'(x)$ is monotonic and $|F'(x)/G(x)|\geq m>0$. Then
$$\left|\int_a^bG(x)e^{iF(x)}dx\right|\leq 4/m.$$
\end{lem}

The lemmas that follow give the Laurent series expansion of the Dirichlet series $F(s,h/k)$ about its singularity, and the functional equation of $F(s,h/k)$ that will be essentially used to study the allied exponential sum, $B(x,h/k)$.
\begin{lem}\label{Fpole}
    The Dirichlet series $F(s,h/k)$ can be continued analytically to a meromorphic function, which is holomorphic in the whole complex plane up to a quadruple pole at $s=1$, and at $s=1$, has the Laurent expansion
\begin{align}
    F(s,h/k)&=k^{-1}(s-1)^{-4}+2k^{-1}\log k\hspace{1mm}(s-1)^{-3}+k^{-1}(\log^{2}k+2\gamma\log k-2\gamma_1)(s-1)^{-2}\notag\\
    &\quad+2k^{-1}(-\log^{3}k+\gamma\log^{2}k-2\gamma_2)(s-1)^{-1}+\ldots,\label{Flaurent}
\end{align}
where the constants $\gamma_1$ and $\gamma_2$ are defined as follows:
\begin{align}
    \zeta(s)=\frac{1}{s-1}+\gamma+\gamma_1(s-1)+\gamma_2(s-1)^2+\ldots.
\end{align}
\end{lem}

\begin{proof}
We first express the function $F(s,h/k)$ in terms of the Hurwitz zeta-function,
\begin{align}
    \zeta(s,a)=\sum_{n=0}^{\infty}(n+a)^{-s} \hspace{0.8cm} (\sigma>1, 0<a\leq1).
\end{align}
One can observe for $\sigma>1$,
\begin{align}
F(s,h/k)&=\sum_{n=1}^{\infty}\bigg(\sum_{d|n}\log d\hspace{0.1cm}\log(n/d)\bigg)e^{\frac{2\pi inh}{k}}n^{-s}\\
    &=\sum_{\alpha,\beta=1}^{k}e_k(\alpha\beta h)\sum_{\substack{m\equiv\alpha(mod\hspace{1mm}k) \\n\equiv\beta(mod\hspace{1mm}k)}}\log m\hspace{1mm}\log n\hspace{1mm}(mn)^{-s}\\
    &=\sum_{\alpha,\beta=1}^{k}e_k(\alpha\beta h)\sum_{\mu,\nu=0}^{\infty}\left(\frac{-\log(\alpha+\mu k)}{(\alpha+\mu k)^s}\right)\left(\frac{-\log(\beta+\nu k)}{(\beta+\nu k)^s}\right)\\
&=k^{-2s}\sum_{\alpha,\beta=1}^{k}e_k(\alpha\beta h)\left[\zeta'(s,\alpha /k)-\log k\hspace{1mm}\zeta(s,\alpha/k)\right]\left[\zeta'(s,\beta/k)-\log k\hspace{1mm}\zeta(s,\beta/k)\right],
\end{align}
so that
\begin{align}
F(s,h/k)=\bigg(k^{-2s}&\sum_{\alpha,\beta=1}^{k}e_k(\alpha\beta h)\left[\zeta'(s,\alpha /k)\hspace{1mm}\zeta'(s,\beta/k)-2\log k\hspace{1mm}\zeta(s,\alpha/k)\hspace{1mm}\zeta'(s,\beta/k)\right]\bigg)\notag\\
    &+\log^2 k \hspace{1mm}E(s,h/k),\label{Fexp}
\end{align}
where $E(s,h/k)$ is as defined in \eqref{E(sh/k)defn}, which can also be expressed as
\begin{align}
E(s,h/k)=k^{-2s}\sum_{\alpha,\beta=1}^{k}e_k(\alpha\beta h)\zeta(s,\alpha/k)\zeta(s,\beta/k),\label{Eexp}
\end{align}
(refer [\cite{JM}, eq. (1.1.5)]). This holds first for $\sigma>1$. But $\zeta(s,a)$ can be continued analytically to a meromorphic function having a simple pole with residue 1 at $s=1$ \cite[p. 37]{TM}, and from this it can be easily deduced that $\zeta'(s,a)$ has an analytic continuation with a  pole of order $2$ having residue $0$ at $s=1$. We also notice that $E(s,h/k)$ has an analytic continuation to a meromorphic function with a double pole at $s=1$ as its only singularity \cite[Lemma 1.1]{JM}. Therefore, by equation \eqref{Fexp} we can conclude that $F(s,h/k)$ an analytic continuation, with its only possible pole of order at most $4$, at $s=1$.
\par Next, to study the behavior of $F(s,h/k)$ near $s=1$, we compare it with the function $G(s,h/k)$ which is defined as
\begin{align}
    G(s,h/k)&=\log^2 k\hspace{1mm}E(s,h/k)-2k^{-2s}\log k\hspace{1mm}\zeta'(s)\sum_{\alpha,\beta=1}^{k}e_k(\alpha\beta h)\zeta(s,\alpha/k)\notag\\
    &\hspace{4cm}+k^{-2s}\zeta'(s)\sum_{\alpha,\beta=1}^{k}e_k(\alpha\beta h)\zeta'(s,\alpha/k)\notag\\
    &=\log^2 k\hspace{1mm}E(s,h/k)-2k^{1-2s}\log k\hspace{1mm}\zeta(s)\zeta'(s)+k^{1-2s}(\zeta'(s))^2,\label{compfunc}
\end{align}
where the second step follows from the fact that
\begin{align}
    \sum_{\alpha,\beta=1}^{k}e_k(\alpha\beta h)\zeta(s,\alpha/k)=k\hspace{0.5mm}\zeta(s) \ \ \mbox{and}\ \ \sum_{\alpha,\beta=1}^{k}e_k(\alpha\beta h)\zeta'(s,\alpha/k)=k\hspace{0.5mm}\zeta'(s),
\end{align}
By equation \eqref{Fexp}, one can see
\begin{align}
 F(s,h/k)-G(s,h/k)=   -2k&^{-2s}\log k\sum_{\beta=1}^{k}\left[\sum_{\alpha=1}^{k}e_k(\alpha\beta h)\zeta(s,\alpha/k)\right](\zeta'(s,\beta/k)-\zeta'(s))\notag\\
    &+k^{-2s}\sum_{\beta=1}^{k}\left[\sum_{\alpha=1}^{k}e_k(\alpha\beta h)\zeta'(s,\alpha/k)\right](\zeta'(s,\beta/k)-\zeta'(s)).\label{diff}
\end{align}
Here, the factor $\zeta'(s,\beta/k)-\zeta'(s)$ is holomorphic at $s=1$ for all $\beta$, and vanishes at $\beta=k$. Moreover, the sum with respect to $\alpha$ is also holomorphic at $s=1$ for $\beta\neq k$, thus the difference of $F(s,h/k)$ and $G(s,h/k)$ is holomorphic at $s=1$. Hence $F(s,h/k)$ and $G(s,h/k)$ have the same principal part, which is equal to that given in \eqref{Flaurent}.
\end{proof}
\begin{lem}\label{funceq}
The function $F(s,h/k)$ satisfies the relation
\begin{align}
k^{2s-1}  F(s,h/k)
&= \mathcal{A}_0(s) F(1-s,-\overline{h}/k)+  \mathcal{B}_0(s) F(1-s,\overline{h}/k) \notag\\
&+ \left(-\mathcal{A}_0^{\prime}(s)+2\mathcal{A}_0(s) \log k \right) F_0(1-s,-\overline{h}/k)\notag\\
&+ \left(-\mathcal{B}_0^{\prime}(s)+2\mathcal{B}_0(s) \log k \right) F_0(1-s,\overline{h}/k)\notag\\
&+\left(\mathcal{A}_1(s)-\log k \mathcal{A}_0^{\prime}(s)
    +\log^2 k \mathcal{A}_0(s) \right)E(1-s,-\overline{h}/k) \notag\\
    & +\left(\mathcal{B}_1(s)-\log k\mathcal{B}_0^{\prime}(s)
    +\log^2 k \mathcal{B}_0(s) \right)E(1-s,\overline{h}/k),\label{funceqexp}
\end{align}
where $\mathcal{A}_i(s)=\frac{(\chi^{(i)}(s))^2}{2}-\frac{(\chi_1^{(i)}(s))^2}{2}$, $\mathcal{B}_i(s)=\frac{(\chi^{(i)}(s))^2}{2}+\frac{(\chi_1^{(i)}(s))^2}{2}$,
\begin{align}
    F_0(s,h/k)&=\sum_{n=1}^{\infty}d_{(0,1)}(n)e(nh/k)n^{-s}\hspace{2mm},\hspace{2mm}(\sigma>1),\\
    &=-\sum_{m,n=1}^{\infty}e_k(mnh)\log m (mn)^{-s}.
\end{align}
Also, we note that
\begin{align}\label{F(0)bound}
    F(0,h/k)\ll k\log^3 k.
\end{align}
\end{lem}

\begin{proof}
Employing the formula given in [\cite{TM}, equation (2.17.3)] and noting \eqref{chichi1}, we have
\begin{align}
    \zeta(s,a)
    &=\chi(s)\sum_{m=1}^{\infty}\frac{\cos(2\pi ma)}{m^{1-s}}\hspace{1mm}+\hspace{1mm}\chi_1(s)\sum_{m=1}^{\infty}\frac{\sin(2\pi ma)}{m^{1-s}},\label{zetaexp}
\end{align}
for $\Re(s)=\sigma<0$, which in turn gives
\begin{align}
 \zeta'(s,a)=&\chi'(s)\sum_{m=1}^{\infty}\frac{\cos(2\pi ma)}{m^{1-s}}+\chi(s)\sum_{m=1}^{\infty}\frac{\cos(2\pi ma)}{m^{1-s}}\log m\hspace{1mm}+\chi_1'(s)\sum_{m=1}^{\infty}\frac{\sin(2\pi ma)}{m^{1-s}}\notag\\
    &+\chi_1(s)\sum_{m=1}^{\infty}\frac{\sin(2\pi ma)}{m^{1-s}}\log m.\label{zeta'exp}
\end{align}
Simplifying the expression for $\zeta'(s,\alpha/k)\zeta'(s,\beta/k)$ and using
\begin{align}
    \sum_{\alpha=1}^{k}e_k(\alpha\beta h\mp m\alpha)=
    \begin{cases}
        k & \text{if $\beta\equiv\pm\hspace{1mm} m\overline{h}$\hspace{1mm}(mod $k$)}\\
        0 & \text{otherwise}
    \end{cases}
\end{align}
we get,
\begin{align}
    &\quad k^{-1}\sum_{\alpha,\beta=1}^{k}e_k(\alpha\beta h)\zeta'(s,\alpha/k)\zeta'(s,\beta/k)\notag\\
    &=\mathcal{A}_1(s)\sum_{m,n=1}^{\infty}e_k(-mn\overline{h})(mn)^{s-1}+\mathcal{B}_1(s)\sum_{m,n=1}^{\infty}e_k(mn\overline{h})(mn)^{s-1}\notag\\
    &+\mathcal{A}_0^{\prime}(s)\sum_{m,n=1}^{\infty}e_k(-mn\overline{h})\log m\hspace{1mm}(mn)^{s-1}+\mathcal{B}_0^{\prime}(s)\sum_{m,n=1}^{\infty}e_k(mn\overline{h})\log m\hspace{1mm}(mn)^{s-1}\notag\\
    &+\mathcal{A}_0(s)\sum_{m,n=1}^{\infty}e_k(-mn\overline{h})\log m\hspace{1mm}\log n\hspace{1mm}(mn)^{s-1}+\mathcal{B}_0(s)\sum_{m,n=1}^{\infty}e_k(mn\overline{h})\log m\hspace{1mm}\log n\hspace{1mm}(mn)^{s-1}\notag\\
    &\hspace{13cm}(\sigma<0).\label{firstexpfunceq}
\end{align}
Similarly, it can be shown that
\begin{align}
    &\quad k^{-1}\sum_{\alpha\beta=1}^{k}e_k(\alpha\beta h)\zeta(s,\alpha/k)\zeta'(s,\beta/k)\notag\\
    &=\frac{\mathcal{A}_0^{\prime}(s)}{2}\sum_{m,n=1}^{\infty}e_k(-mn\overline{h})(mn)^{s-1}+\frac{\mathcal{B}_0^{\prime}(s)}{2}\sum_{m,n=1}^{\infty}e_k(mn\overline{h})(mn)^{s-1}\notag\\
    &+\mathcal{A}_0(s)\sum_{m,n=1}^{\infty}e_k(-mn\overline{h})\log m\hspace{1mm}(mn)^{s-1}+\mathcal{B}_0(s)\sum_{m,n=1}^{\infty}e_k(mn\overline{h})\log m\hspace{1mm}(mn)^{s-1}\hspace{7mm}(\sigma<0),\label{secondexpfunceq}
\end{align}
and
\begin{align}
    &\quad k^{-1}\sum_{\alpha\beta=1}^{k}e_k(\alpha\beta h)\zeta(s,\alpha/k)\zeta(s,\beta/k)\notag\\
    &=\mathcal{A}_0(s)\sum_{m,n=1}^{\infty}e_k(-mn\overline{h})(mn)^{s-1}+\mathcal{B}_0(s)\sum_{m,n=1}^{\infty}e_k(mn\overline{h})(mn)^{s-1}\hspace{14mm}(\sigma<0).\label{thirdexpfunceq}
\end{align}
Using \eqref{Eexp} in equation \eqref{Fexp} and substituting from \eqref{firstexpfunceq}, \eqref{secondexpfunceq} and \eqref{thirdexpfunceq}, the formula \eqref{funceqexp} follows for $\sigma<0$, but can be continued analytically to the whole complex plane.\\
For a proof of \eqref{F(0)bound}, we derive an expression for $F(0,h/k)$ in a closed form. By \eqref{Fexp},
\begin{align}
    F(0,h/k)=\sum_{\alpha,\beta=1}^{k}e_k(\alpha\beta h)\left[\zeta'(0,\alpha /k)\zeta'(0,\beta/k)\right.&-2\log k\hspace{1mm}\zeta(0,\alpha/k)\zeta'(0,\beta/k)\notag\\
    &\left.+\log^{2}k\hspace{1mm}\zeta(0,\alpha/k)\zeta(0,\beta/k)\right].\label{F(0)exp}
\end{align}
If $0<a<1$, then the series in \eqref{zetaexp} converges uniformly and thus defines a continuous function for all real $\sigma\leq 0$. By continuity, \eqref{zetaexp} remains valid for $s=0$ also. Thus, we obtain
\begin{align}\label{zeta(0,a)}
    \zeta(0,a)=\pi^{-1}\sum_{m=1}^{\infty}\sin{(2\pi ma)}m^{-1}.
\end{align}
Since the series on the right equals $\pi(1/2-a)$ for $0<a<1$, thus
\begin{align}\label{zeta(0,a)2}
    \zeta(0,a)=1/2-a.
\end{align}
Also, by [\cite{WR}, Equation (16)]
\begin{align}\label{zeta'(0,a)}
    \zeta'(0,a)=\log\frac{\Gamma(a)}{\sqrt{2\pi}}.  
\end{align}
Since $\zeta(0,1)=\zeta(0)=-1/2$ and $\zeta'(0,1)=\zeta'(0)=-\frac{1}{2}\log{2\pi}$, \eqref{zeta(0,a)2} and \eqref{zeta'(0,a)} hold for $a=1$ as well.
Now, by \eqref{F(0)exp}, \eqref{zeta(0,a)2} and \eqref{zeta'(0,a)},
\begin{align}
    F(0,h/k)=\sum_{\alpha,\beta=1}^{k}e_k(\alpha\beta h)\left[\log\frac{\Gamma(\alpha/k)}{\sqrt{2\pi}}\log\frac{\Gamma(\beta/k)}{\sqrt{2\pi}}\right.&-2\log k\hspace{1mm}(1/2-\alpha/k)\log\frac{\Gamma(\beta/k)}{\sqrt{2\pi}}\\
    &+\log^2 k\hspace{1mm}(1/2-\alpha/k)(1/2-\beta/k)\bigg].
\end{align}
From this it follows that
\begin{align}
    F(0,h/k)=\frac{k}{2}\log{2\pi}\log&\frac{\sqrt{2\pi}}{k}-\frac{3}{4}k\log^2 k+\sum_{\alpha,\beta=1}^{k}e_k(\alpha\beta h)\bigg[\log\Gamma(\alpha/k)\log\Gamma(\beta/k)\notag\\
    &+\frac{2\log k}{k}\alpha\log\Gamma(\beta/k)-\log{(2\pi k)}\log\Gamma(\beta/k)+\frac{\log^2 k}{k^2}\alpha\beta\bigg].\label{F(0)exp2}
\end{align}
To estimate the double sums on the right, note that if $1\leq \alpha\leq k-1$ and $\beta$ runs over an arbitrary interval, then using \cite[Theorem 2.1, p. 7]{VC}
\begin{align}
    \left|\sum_{\beta}e_k(\alpha \beta h)\right|\ll ||\alpha h/k||^{-1}.\label{e_kpartialsumbound}
\end{align}
Thus, by repeated use of partial summation,
\begin{align}
    \left|\sum_{\alpha=1}^{k-1}\sum_{\beta=1}^{k}e_k(\alpha\beta h)\log\Gamma(\alpha/k)\log\Gamma(\beta/k)\right|&\ll \log k\sum_{\alpha=1}^{k-1}\log\Gamma(\alpha/k)\hspace{1mm}||\alpha h/k||^{-1}
    \ll k\log^3 k.\label{F(0)firstsum}
\end{align}
Similarly, we can have
\begin{align}
\left|\sum_{\alpha=1}^{k}\sum_{\beta=1}^{k}e_k(\alpha\beta h)\alpha\log\Gamma(\beta/k)\right|&\leq\left|\sum_{\alpha=1}^{k-1}\alpha\sum_{\beta=1}^{k}e_k(\alpha\beta h)\log\Gamma(\beta/k)\right|+k\left|\sum_{\beta=1}^{k}\log\Gamma(\beta/k)\right|\notag\\
    &\ll \log k\sum_{\alpha=1}^{k-1}\alpha||\alpha h/k||^{-1}+k(k+\log k)\notag\\
    &\ll k^2 \log k.\label{F(0)secondsum}
\end{align}
Also, the bounds
\begin{align}\label{F(0)thirdsum}
\left|\sum_{\alpha=1}^{k}\sum_{\beta=1}^{k}e_k(\alpha\beta h)\log\Gamma(\beta/k)\right|\ll k\log k,
\end{align}
and
\begin{align}\label{F(0)fourthsum}
\left|\sum_{\alpha=1}^{k}\sum_{\beta=1}^{k}e_k(\alpha\beta h)\alpha\beta\right|\ll k^3
\end{align}
can be obtained similarly. Thus, \eqref{F(0)bound}  follows from \eqref{F(0)exp2}-\eqref{F(0)fourthsum}.
\end{proof}

Using the expressions for $\chi(s)$ and $\chi_1(s)$ in Lemma \eqref{funceq}, the functional equation for $F(s,h/k)$ can be alternatively expressed as 
\begin{align}
    &F(s,h/k)=2(2\pi)^{2s-2}k^{1-2s} \Gamma^2(1-s) \bigg\{\bigg[F(1-s,\overline{h}/k)-\cos(\pi s)\hspace{1mm}F(1-s,-\overline{h}/k)\bigg]\notag\\
    &+\bigg[\bigg(F_0(1-s,\overline{h}/k)-\cos(\pi s)\hspace{1mm}F_0(1-s,-\overline{h}/k)\bigg)\bigg(2\psi(1-s)-\log\left(\frac{4\pi^2}{k^2}\right)\bigg)-\pi\sin(\pi s)\times\notag\\
    &\hspace{11.5cm}\times F_0(1-s,-\overline{h}/k)\bigg]\notag\\
    &+\bigg[\bigg(E(1-s,\overline{h}/k)-\cos(\pi s)\hspace{1mm}E(1-s,-\overline{h}/k)\bigg)\bigg(\log^2(2\pi)+\psi^2(1-s)-\log \left(\frac{4\pi^2}{k^2}\right)\psi(1-s)\notag\\
    &\hspace{11.7cm}-\log\left(\frac{4\pi^2}{k}\right)\log k\bigg)\notag\\
    &+\frac{\pi^2}{4}\bigg(E(1-s,\overline{h}/k)+\cos(\pi s)\hspace{1mm}E(1-s,-\overline{h}/k)\bigg)+\pi \sin(\pi s)\hspace{1mm}E(1-s,-\overline{h}/k)\times\notag\\
    &\hspace{9.5cm}\times\bigg(\log\left(\frac{2\pi}{k}\right)-\psi(1-s)\bigg)\bigg]\bigg\}.\label{altfunceq}
\end{align}
where $\psi(s)$ is the digamma function.\\
While the representation given in \eqref{funceqexp} will be helpful to calculate the exact terms in the formulas of $\Delta_a(x,h/k)$ that will follow, it will be more convenient to use \eqref{altfunceq} in estimating the bound for $F(s,h/k)$.


\section{Proof of Theorem \ref{th1}}\label{proofofth1}
   This section is devoted to the proof of Theorem \ref{th1}.\\
    Let $\delta>0$ be an arbitrary real number. By using the version of Perron's formula [\cite{TM}, p. 60, Lemma 3.12] over a finite contour of integration, we have
    \begin{align}\label{perronform}
        B(x,h/k)=\frac{1}{2\pi i}\int_{1+\delta-iT}^{1+\delta+iT}F(s,h/k)\frac{x^s}{s}ds\hspace{1mm}+\hspace{1mm}O\bigg(\frac{x^{1+\delta}}{T}\bigg),
    \end{align}
    where $T$ is a parameter such that 
    \begin{align}\label{Tcondition}
        3\leq T\ll k^{-1}x.
    \end{align}
     Next, we will move the integration in \eqref{perronform} to the line $\sigma=-\delta$ in order to evaluate the integral around the rectangle with vertices $1+\delta+iT$, $-\delta+iT$, $-\delta-iT$ and $1+\delta-iT$. Since $F(s,h/k)$ has a pole of order $4$ at $s=1$, to find an estimate for $F(s,h/k)$ in the strip $-\delta\leq \Re(s)=\sigma\leq 1+\delta$ for $|t|\geq 3$, we define an auxiliary function
    \begin{align}\label{auxfunc}
    \left(\frac{s-1}{s-2}\right)^4F(s,h/k)
    \end{align}
    which is holomorphic in the strip $-\delta\leq\sigma\leq 1+\delta$. 
    Note that the function in \eqref{auxfunc} is of the same order of magnitude as $F(s,h/k)$. By definition, $D_{1}(n)\ll d(n)\log^2(n)$. Hence,
    \begin{align}
        \left|\left(\frac{s-1}{s-2}\right)^4F(s,h/k)\right|\leq \left(1+\frac{1}{|s-2|}\right)^4\sum_{n=1}^{\infty}\frac{|D_{(1)}(n)|}{n^{\sigma}},
    \end{align}
    that is, the auxiliary function is bounded on the line $\sigma=1+\delta$.\\
    On the line $\sigma=-\delta$, by using the estimate of the gamma function in \eqref{gammasmallestimate}, we get the following estimate for $A(s)=2(2\pi)^{2s-2}k^{1-2s} \Gamma(1-s)$ in \eqref{altfunceq},
    \begin{align}
        |A(s)|&\ll 2(2\pi)^{2\sigma-2}k^{1-2\sigma}|t|^{1-2\sigma}e^{-\pi|t|}
        \ll (k|t|)^{1+2\delta}e^{-\pi|t|}.\label{A(s)bound}
    \end{align}
    Clearly, $F(1-s,\pm \overline{h}/k)$, $F_0(1-s,\pm \overline{h}/k)$ and $E(1-s, \pm \overline{h}/k)$ are all bounded on the line $\sigma=-\delta$. Moreover, we have
    \begin{align}\label{sinestimate}
        \sin(\pi s)=\frac{e^{\pm i\pi/2}}{2}e^{\pi |t|}e^{\mp i\pi\sigma}(1+O(1/|t|))\hspace{2mm},\hspace{1mm}|t|\geq 1,
    \end{align}
    and also 
    \begin{align}\label{cosestimate}
        \cos(\pi s)=\frac{1}{2}e^{\pi |t|}e^{\mp i\pi\sigma}(1+O(1/|t|))\hspace{2mm},\hspace{1mm}|t|\geq 1.
    \end{align}
    Also,  using  \cite[equation (6)), p.73]{DP}, we estimate the digamma function as
    \begin{align}
        \psi(1-s)
        &=\log\left(1+\frac{i(1+\delta)}{t}\right)+\log|t|+i3\pi/2+O(1/|t|)
        \ll \log|t|,\label{digammaestimate}
    \end{align}
    for $|t|\geq 1$. 
Thus, combining \eqref{A(s)bound}, \eqref{sinestimate}, \eqref{cosestimate} together with \eqref{digammaestimate} in \eqref{altfunceq}, we obtain
\begin{align}
    \left(\frac{s-1}{s-2}\right)^4F(s,h/k)&\ll (k|t|)^{1+2\delta}(k|t|)^{\delta}
    \ll (k(|t|+1))^{1+3\delta}\hspace{5mm} at \hspace{2mm}\sigma=-\delta.
\end{align}
We now use generalization of the Phragme\'n Lindelo\"f theorem [\cite{HR}, Section 33, p. 66] for calculating an estimate of the function in \eqref{auxfunc}, and consequently an estimate for $F(s,h/k)$ in the region $-\delta\leq\sigma\leq 1+\delta$, $|t|\geq 3$. Hence, we get
\begin{align}
    \left|\left(\frac{s-1}{s-2}\right)^{4}F(s,h/k)\right|\ll (k|t|)^{\frac{(1+3\delta)(1+\delta-\sigma)}{1+2\delta}},
\end{align}
which in turn will give
\begin{align}\label{Fbound}
     |F(s,h/k)|\ll (k|t|)^{1-\sigma+2\delta} \hspace{3mm} for -\delta\leq\sigma\leq 1+\delta \hspace{1mm}, \hspace{1mm}|t|\geq 3.
\end{align}
Let $C$ be the rectangular contour with vertices $1+\delta\pm iT$ and $-\delta\pm iT$. Then by the Cauchy's residue theorem,
\begin{align}
 I:=   \frac{1}{2\pi i}\int_C F(s,h/k)\hspace{1mm}x^s s^{-1}ds&=\underset{s=0}{Res}\hspace{1mm}(F(s,h/k)\hspace{1mm}x^s s^{-1})+\underset{s=1}{Res}\hspace{1mm}(F(s,h/k)\hspace{1mm}x^s s^{-1}).
\end{align}
To calculate the residue at $s=1$, we use the expansion in \eqref{Flaurent},
\begin{align}
   (s-1)^4 F(s,h/k)=&\frac{1}{k}+\frac{\log k}{k}(s-1)+\frac{1}{k}(\log ^2 k+2\gamma\log k-2\gamma_1)(s-1)^2\\
    &+\frac{2}{k}(-\log^3 k+\gamma\log^2 k-2\gamma_2)(s-1)^3+O(|s-1|^4)+\cdots.
\end{align} Thus the integral becomes
\begin{align}
   I
    =&\hspace{1mm}\frac{x}{k} \left[\frac{1}{6}\log^3 x+\left(\log k-\frac{1}{2}\right)\log^2 x
    +(\log^2 k+2(\gamma-1)\log k-2\gamma_1+1)\log x \right.\\
    &+(-2\log^3 k+(2\gamma-1)\log^2 k-2(\gamma-1)\log k+2\gamma_1-4\gamma_2-1) \bigg]+F(0,h/k).\label{Fintformula}
\end{align}
The integrals over the horizontal parts of $C$ are obtained as
\begin{align}
    \int_{-\delta\pm iT}^{1+\delta\pm iT}F(s,h/k)x^s s^{-1}ds
    &\ll \int_{-\delta}^{1+\delta}(kT)^{1-\sigma+2\delta}x^{\sigma}\hspace{1mm}T^{-1}d\sigma \ll x^{1+2\delta}\hspace{1mm}T^{-1} \label{bound11}
\end{align}
by using \eqref{Fbound} and $T\ll k^{-1}x$. 
Employing \eqref{bound11} and \eqref{Fintformula} in \eqref{perronform}, we obtain
\begin{align}
    B(x,h/k)=&k^{-1}x \left[ \frac{1}{6}\right. \log^3 x+\left(\log k-\frac{1}{2}\right)\log^2 x+(\log^2 k+2(\gamma-1)\log k-2\gamma_1+1)\log x\notag\\
    &+(-2\log^3 k+(2\gamma-1)\log^2 k-2(\gamma-1)\log k+2\gamma_1-4\gamma_2-1) \bigg]\notag\\
    &+F(0,h/k)+\Delta(x,h/k), \label{B_0formula}
\end{align}
where
\begin{align}
    \Delta(x,h/k)=\frac{1}{2\pi i}\int_{-\delta-iT}^{-\delta+iT}F(s,h/k)x^s\hspace{1mm}s^{-1}ds+O(x^{1+2\delta}\hspace{1mm}T^{-1}).
\end{align}
At $\sigma=-\delta$, $F(-\delta+it,h/k)$ is analytic for $|t|\leq 3$. Employing the fact that $F(s,h/k)$ is bounded for $|t|\leq 3$, we have
\begin{align}
    \int_{-\delta-3i}^{-\delta+3i}F(s,h/k)x^s\hspace{1mm}s^{-1}ds
    \ll x^{1+2\delta}\hspace{1mm}T^{-1},
\end{align}
as $k\ll x/T$.\\
To avoid complicated integral expressions, we shall set a notation. Throughout this paper, we shall write
\begin{align} \label{branchedintegral}
 \int_{b-iT}^{b+iT}\llap{=\quad\hspace{11.5pt}}=\int_{b-iT}^{b-3i}+\int_{b+3i}^{b+iT} \ \mbox{and} \ \ \int_{-T}^{T}\llap{=\quad\hspace{0.1pt}}=\int_{-T}^{-3}+\int_{3}^{T}.
 \end{align}
Thus,  we have
\begin{align}\label{deltaform}
\Delta(x,h/k)=\frac{1}{2\pi i}\int_{-\delta-iT}^{-\delta+iT}\llap{=\quad\quad\hspace{8pt}}F(s,h/k)x^s\hspace{1mm}s^{-1}ds+O(x^{1+2\delta}\hspace{1mm}T^{-1}).\end{align}
We now apply the relation \eqref{funceqexp} to \eqref{deltaform}. Observe that the coefficients of the terms involving $F(1-s,\overline{h}/k)$, $F_0(1-s,\overline{h}/k)$ and $E(1-s,\overline{h}/k)$ contain $e^{-\pi|t|}$, and hence decrease rapidly as $|t|$ increases. Thus, these terms will be estimated as part of the error term.
By \eqref{altfunceq} and \eqref{A(s)bound}, the contribution of these error terms to the integral in equation \eqref{deltaform} is 

$$\int_{-\delta-iT}^{-\delta+iT}\llap{=\quad\quad\hspace{8pt}}(k(|t|+1))^{1+3\delta}e^{-\pi |t|}x^s s^{-1}ds \ll k^{1+3\delta}x^{-\delta}\int_{-T}^{T}\llap{=\quad\hspace{0.1pt}}(1+|t|)^{3\delta}e^{-\pi |t|}dt
\ll x^{1+2\delta}T^{-1}.$$
Thus, we have
\begin{align}
&\Delta(x,h/k)=k\sum_{n=1}^{\infty}\frac{d(n)\hspace{1mm}e_k(-n\overline{h})}{n}\hspace{1mm}\frac{1}{2\pi i}\int_{-\delta-iT}^{-\delta+iT}\llap{=\quad\quad\hspace{8pt}}\bigg[\frac{(\chi'(s))^2}{2}-\frac{(\chi_1'(s))^2}{2}-\log k\hspace{1mm}\chi(s)\chi'(s)\notag\\
    &\hspace{3.5cm}+\log k\hspace{1mm}\chi_1(s)\chi_1'(s)+\log^2 k\frac{(\chi(s))^2}{2}-\log^2 k\frac{(\chi_1(s))^2}{2}\bigg]\left(\frac{nx}{k^2}\right)^s s^{-1}ds\notag\\
&\hspace{1.7cm}+k\sum_{n=1}^{\infty}\frac{d_{(0,1)}(n)\hspace{1mm}e_k(-n\overline{h})}{n}\hspace{1mm}\frac{1}{2\pi i}\int_{-\delta-iT}^{-\delta+iT}\llap{=\quad\quad\hspace{8pt}}\left[-\chi(s)\chi'(s)+\chi_1(s)\chi_1'(s)\right.\notag\\
    &\hspace{7cm}\left.+\log k(\chi(s))^2-\log k(\chi_1(s))^2\right]\left(\frac{nx}{k^2}\right)^s s^{-1}ds\notag\\
    &+k\sum_{n=1}^{\infty}\frac{D_{(1)}(n)\hspace{1mm}e_k(-n\overline{h})}{n}\hspace{1mm}\frac{1}{2\pi i}\int_{-\delta-iT}^{-\delta+iT}\llap{=\quad\quad\hspace{8pt}}\left[\frac{(\chi(s))^2}{2}+\frac{(\chi_1(s))^2}{2}\right]\left(\frac{nx}{k^2}\right)^s s^{-1}ds+O(x^{1+2\delta}\hspace{1mm}T^{-1}).
\end{align}
Since the integrals and sums are absolutely convergent, hence we interchange the order of integration and summation to obtain
\begin{align}
&\Delta(x,h/k)=\hspace{1mm}k\sum_{n=1}^{\infty}\frac{d(n)}{2n}e_k(-n\overline{h})\hspace{1mm}\frac{1}{2\pi i}\int_{-\delta-iT}^{-\delta+iT}\llap{=\quad\quad\hspace{8pt}}(\chi'(s))^2\left(\frac{nx}{k^2}\right)^s s^{-1}ds\notag\\
    &\hspace{1.3cm}-k\sum_{n=1}^{\infty}\frac{d(n)}{2n}e_k(-n\overline{h})\hspace{1mm}\frac{1}{2\pi i}\int_{-\delta-iT}^{-\delta+iT}\llap{=\quad\quad\hspace{8pt}}(\chi_1'(s))^2\left(\frac{nx}{k^2}\right)^s s^{-1}ds\notag\\
    &\hspace{1.3cm}-k\sum_{n=1}^{\infty}\frac{(d(n)\log k+d_{(0,1)}(n))}{n}e_k(-n\overline{h})\hspace{1mm}\frac{1}{2\pi i}\int_{-\delta-iT}^{-\delta+iT}\llap{=\quad\quad\hspace{8pt}}\chi(s)\chi'(s)\left(\frac{nx}{k^2}\right)^s s^{-1}ds\notag\\
&\hspace{1.3cm}+k\sum_{n=1}^{\infty}\frac{(d(n)\log k+d_{(0,1)}(n))}{n}e_k(-n\overline{h})\hspace{1mm}\frac{1}{2\pi i}\int_{-\delta-iT}^{-\delta+iT}\llap{=\quad\quad\hspace{8pt}}\chi_1(s)\chi_1'(s)\left(\frac{nx}{k^2}\right)^s s^{-1}ds\notag\\
&+k\sum_{n=1}^{\infty}\frac{(d(n)\log^2 k+2d_{(0,1)}(n)\log k+D_{(1)}(n))}{2n}e_k(-n\overline{h})\frac{1}{2\pi i}\int_{-\delta-iT}^{-\delta+iT}\llap{=\quad\quad\hspace{8pt}}(\chi(s))^2\left(\frac{nx}{k^2}\right)^s s^{-1}ds\notag\\
    &-k\sum_{n=1}^{\infty}\frac{(d(n)\log^2 k+2d_{(0,1)}(n)\log k+D_{(1)}(n))}{2n}e_k(-n\overline{h})\frac{1}{2\pi i}\int_{-\delta-iT}^{-\delta+iT}\llap{=\quad\quad\hspace{8pt}}(\chi_1(s))^2\left(\frac{nx}{k^2}\right)^s s^{-1}ds\notag\\    &\hspace{11cm}+O(x^{1+2\delta}\hspace{1mm}T^{-1}).\label{deltafinalform}
\end{align}
At this stage, we fix the parameter $T$ by putting 
\begin{align}\label{Tparameter}
T^2 k^2 (4\pi^2 x)^{-1}=N+\frac{1}{2}\hspace{1mm},
\end{align}
where $N$ is an integer such that $1\leq N\ll x$. Hence, it readily follows that $T\ll k^{-1}x$. But for condition \eqref{Tcondition} to hold, we also need $T\geq 3$, which presupposes that $N\gg k^2 \hspace{1mm}x^{-1}$.\\
Note that if $1\leq N \ll k^2x^{-1}$, then \eqref{errorformula} is implied by the estimate $\Delta(x,h/k)\ll x^{1+\varepsilon}$, which follows from \eqref{B_0formula} and \eqref{F(0)bound}. Hence, we may assume that $N\gg k^2x^{-1}$, else the assertion \eqref{errorformula} holds trivially. \\
Next, we consider the tail $n>N$ of all the series in equation \eqref{deltafinalform}, and show that it will be a part of the error term. The integrals split into two parts, namely $t$ running over the intervals $[-T,-3]$ and $[3,T]$ respectively. Since integrals over both the intervals are similar, consider the second integral, say $I_1$, where
\begin{align}\label{I12}
    I_1=\frac{1}{2\pi i}\int_{-\delta+i3}^{-\delta+iT}(\chi'(s))^2\left(\frac{nx}{k^2}\right)^{s}s^{-1}ds.
\end{align}
By \eqref{gammabigestimate}, 
\begin{align}
    \chi(s)=2(2\pi)^{s-1/2}\hspace{1mm}t^{1/2-s}\sin\left(\frac{\pi s}{2}\right)exp\left[-\frac{\pi |t|}{2}+it-\frac{i\pi}{2}\left(\frac{1}{2}-\sigma\right)\right](1+O(1/t)).\label{chiformula}
\end{align}
Since 
\begin{align}
\chi'(s)=\chi(s)(\log(2\pi)-\psi(1-s))+\frac{\pi}{2}\hspace{1mm}\chi_1(s),\label{chi'formula}
\end{align}
hence
\begin{align}
    (\chi'(s))^2=4(2\pi)^{2s-1}t^{1-2s}&exp\left[-\pi |t|+2it-i\pi\left(\frac{1}{2}-\sigma\right)\right](1+O(1/t))\hspace{1mm}\nonumber \\
    &\times\left((\log^2 2\pi +\psi^2(1-s)-\log(4\pi^2)\psi(1-s))\sin^2\left(\frac{\pi s}{2}\right)\right.\\
    &\hspace{1cm}\left.+\frac{\pi^2}{4}\cos^2\left(\frac{\pi s}{2}\right)+\frac{\pi}{2}(\log 2\pi-\psi(1-s))\sin \pi s\right).\label{chi'square}
\end{align}
Let 
\begin{align}\label{Fdefn}
F(t)=-2t\log|t|+2t+t\log(4\pi^2nxk^{-2}).
\end{align}
Consider first,
\begin{align}
    L_{11}=\int_{3}^{T}\log^2(2\pi)\frac{t}{s}\hspace{1mm}t^{2\delta}e^{-\pi|t|}e^{-i\pi(\delta+1/2)}e^{iF(t)}\sin^2(\pi s/2)\left(\frac{4\pi^2nx}{k^2}\right)^{-\delta}(1+O(1/t))\hspace{1mm}dt.
\end{align}
Using \eqref{sinestimate}, one can observe that
\begin{align}\label{modI1}
 |L_{11}|\leq \frac{\log^2(2\pi)}{4}\hspace{1mm}\left(\frac{4\pi^2nx}{k^2}\right)^{-\delta}\left[\hspace{1mm}\left|\int_3^Tt^{2\delta}e^{iF(t)}dt\right|+\left|\int_3^Tt^{2\delta}e^{iF(t)}O(1/t)dt\right|\hspace{1mm}\right].
\end{align}
Clearly, the second integral is $\ll T^{2\delta}$. For the first integral in \eqref{modI1}, we take $G(t)=t^{2\delta}$ in Lemma \eqref{expintegrallemma} together with the fact 
$$\left|\frac{F'(t)}{G(t)}\right|=t^{-2\delta}\left|\log\left(\frac{4\pi^2nx}{k^2t^2}\right)\right|\geq T^{-2\delta}\log\left(\frac{n}{N+1/2}\right)$$
to obtain
$$\left|\int_3^Tt^{2\delta}e^{iF(t)}dt\right|\leq 4T^{2\delta}\left(\log\left(\frac{n}{N+1/2}\right)\right)^{-1}.$$
Thus substituting the above bound in \eqref{modI1}, we get
\begin{align}\label{I1}
    L_{11}\ll n^{-\delta}x^{-\delta}k^{2\delta}T^{2\delta}\left(\left(\log\left(\frac{n}{N+1/2}\right)\right)^{-1}+1\right).
\end{align}
Next consider
$$L_{12}=\int_{3}^{T}\frac{t}{s}\hspace{1mm}t^{2\delta}e^{-\pi|t|}e^{-i\pi(\delta+1/2)}e^{iF(t)}\psi^2(1-s)\sin^2(\pi s/2)\left(\frac{4\pi^2nx}{k^2}\right)^{-\delta}(1+O(1/t))\hspace{1mm}dt,$$
Using the fact $\psi^2(1-s)=\log^2(1-s)+O\left(\frac{\log|t|}{|t|}\right),$ the above integral can be estimated as
\begin{align}\label{modI2}
    |L_{12}|\leq \frac{1}{4}\left(\frac{4\pi^2nx}{k^2}\right)^{-\delta}\left[\left|\int_3^T t^{2\delta}\log^2(1-s)e^{iF(t)}dt\right|+O\left(\int_3^T \frac{t^{2\delta}\log^2 t}{t}dt\right)\right].
\end{align}
Taking into account $\log^2(1-s)
=\log^2t+i3\pi\log t-9\pi^2/4+O\left(\frac{\log t}{t}\right)$, from \eqref{modI2} we obtain 
\begin{align}
    |L_{12}|\leq \frac{1}{4}\left(\frac{4\pi^2nx}{k^2}\right)^{-\delta}\left[\left|\int_3^T t^{2\delta}\log^2 t\hspace{1mm} e^{iF(t)}dt\right|\right.&\left.+3\pi\left|\int_3^T t^{2\delta}\log t\hspace{1mm} e^{iF(t)}dt\right|\right.\notag\\
    &\left.+\frac{9\pi^2}{4}\left|\int_3^T t^{2\delta}\hspace{1mm}e^{iF(t)}dt\right|+O(T^{3\delta})\right].\label{modI2final}
\end{align}
For the first integral above, taking $G(t)=t^{2\delta}\log^2t$ in Lemma \eqref{expintegrallemma}, together with the fact
\begin{align}
    \left|\frac{F'(t)}{G(t)}\right|\geq \frac{1}{T^{2\delta}\log^2 T}\log\left(\frac{n}{N+1/2}\right)
\end{align}
by \eqref{Tparameter}, we obtain that the first integral in \eqref{modI2final} is   $$\leq 4\hspace{1mm}T^{2\delta}\log^2 T\left(\log\left(\frac{n}{N+1/2}\right)\right)^{-1}.$$
Similarly, we may show that the other two integrals in \eqref{modI2final} are bounded by
$$\ll \hspace{1mm}T^{2\delta}\log T\left(\log\left(\frac{n}{N+1/2}\right)\right)^{-1}.$$
Hence, \eqref{modI2final} becomes
\begin{align}\label{I2estimate}
    L_{12}\ll n^{-\delta}x^{-\delta}k^{2\delta}T^{3\delta}\left(\left(\log\left(\frac{n}{N+1/2}\right)\right)^{-1}+1\right).
\end{align}
The other terms in \eqref{chi'square} give rise to the following integrals :-
$$L_{13}=\int_{3}^{T}\log (4\pi^2)\frac{t}{s}\hspace{1mm}t^{2\delta}e^{-\pi|t|}e^{-i\pi(\delta+1/2)}e^{iF(t)}\psi(1-s)\sin^2(\pi s/2)\left(\frac{4\pi^2nx}{k^2}\right)^{-\delta}(1+O(1/t))\hspace{1mm}dt,$$
$$L_{14}=\frac{\pi}{2}\int_{3}^{T}\frac{t}{s}\hspace{1mm}t^{2\delta}e^{-\pi|t|}e^{-i\pi(\delta+1/2)}e^{iF(t)}(\log(2\pi)-\psi(1-s))\sin(\pi s)\left(\frac{4\pi^2nx}{k^2}\right)^{-\delta}(1+O(1/t))\hspace{1mm}dt,$$
and 
$$L_{15}=\frac{\pi^2}{4}\int_{3}^{T}\frac{t}{s}\hspace{1mm}t^{2\delta}e^{-\pi|t|}e^{-i\pi(\delta+1/2)}e^{iF(t)}\cos^2(\pi s/2)\left(\frac{4\pi^2nx}{k^2}\right)^{-\delta}(1+O(1/t))\hspace{1mm}dt.$$
Employing Lemma \eqref{expintegrallemma} repeatedly for every integral, 
we may show that $L_{13}$, $L_{14}$ and $L_{15}$ are
\begin{align}\label{I5}
\ll n^{-\delta}x^{-\delta}k^{2\delta}T^{2\delta}\left(\left(\log\left(\frac{n}{N+1/2}\right)\right)^{-1}+1\right).
\end{align}
Hence, using \eqref{I1}, \eqref{I2estimate} and \eqref{I5} in \eqref{I12}, we get
\begin{align}
    I_1\ll n^{-\delta}x^{-\delta}k^{2\delta}T^{3\delta}\left(\left(\log\left(\frac{n}{N+1/2}\right)\right)^{-1}+1\right).
\end{align}
Hence, we obtain
\begin{align}
  &  k\sum_{n>N}\frac{d(n)}{2n}e_k(-n\overline{h})\hspace{1mm}\frac{1}{2\pi i}\int_{-\delta-iT}^{-\delta+iT}\llap{=\quad\quad\hspace{8pt}}(\chi'(s))^2\left(\frac{nx}{k^2}\right)^s s^{-1}ds\\
    &\ll x^{-\delta}k^{1+2\delta}T^{3\delta}\sum_{n>N}\frac{1}{n^{1+\frac{\delta}{2}}}\left(1+\frac{1}{\log n}\left(1+O\left(\frac{\log (N+1/2)}{\log n}\right)\right)\right)\\
    &\ll k\hspace{1mm}T^{\delta}\left(N+\frac{1}{2}\right)^{\delta}\ll k\hspace{1mm}x^{2\delta},
\end{align}
using the facts \eqref{Tparameter}, $T\ll k^{-1}x$ and $N\ll x$. Next, we consider $$I_2=\frac{1}{2\pi i}\int_{-\delta+i3}^{-\delta+iT}\chi(s)\chi'(s)\left(\frac{nx}{k^2}\right)^s s^{-1}ds.$$
By \eqref{chiformula} and \eqref{chi'formula}, we get
\begin{align}
    \chi(s)\chi'(s)&=\chi^{2}(s)(\log (2\pi)-\psi(1-s))+\frac{\pi}{2}\chi(s)\chi_1(s)\notag\\
    &= (2\pi)^{2s-1}t^{1-2s}exp\left[-\pi|t|+2it-i\pi\left(\frac{1}{2}-\sigma\right)\right](1+O(1/t))\notag\\
    &\hspace{24mm}\left[4\sin^2\left(\frac{\pi s}{2}\right)(\log(2\pi)-\psi(1-s))+\pi\sin(\pi s)\right].\label{chichi'}
\end{align}
Let $F(t)$ be as defined in \eqref{Fdefn}, then
\begin{align}
    I_2=\int_{3}^{T}\frac{t}{s}\hspace{1mm}t^{2\delta}e^{-\pi|t|}e^{-i\pi(\delta+1/2)}e^{iF(t)}((\log(2\pi)-\psi(1&-s))4\sin^2(\pi s/2)+\pi\sin(\pi s))\times\\
    &\hspace{0.7cm}\times\left(\frac{4\pi^2nx}{k^2}\right)^{-\delta}
    (1+O(1/t))\hspace{1mm}dt.
\end{align}
Using 
Lemma \eqref{expintegrallemma}, we may show that
$$I_2\ll n^{-\delta}x^{-\delta}k^{2\delta}T^{3\delta}\left(\left(\log\left(\frac{n}{N+1/2}\right)\right)^{-1}+1\right),$$
as done in the previous case. Then,
\begin{align}
   & k\sum_{n>N}\frac{(d(n)\log k+d_{(0,1)}(n)) e_k(-n\overline{h})}{n}\frac{1}{2\pi i}\int_{-\delta-iT}^{-\delta+iT}\llap{=\quad\quad\hspace{8pt}}\chi(s)\chi'(s)\left(\frac{nx}{k^2}\right)^s s^{-1}ds\\
   & \ll x^{-\delta}k^{1+3\delta}T^{3\delta}\sum_{n>N}\frac{1}{n^{1+\frac{\delta}{2}}}
    \ll k\hspace{1mm}x^{2\delta},
\end{align}
as shown in the previous case. Next, we consider
$$I_3=\frac{1}{2\pi i}\int_{-\delta+i3}^{-\delta+iT}\llap{=\quad\quad\hspace{8pt}}(\chi(s))^2\left(\frac{nx}{k^2}\right)^s s^{-1}ds.$$
Using \eqref{chiformula} and proceeding as before, we get
\begin{align}
    I_3\ll n^{-\delta}x^{-\delta}k^{2\delta}T^{2\delta}\left(\left(\log\left(\frac{n}{N+1/2}\right)\right)^{-1}+1\right).
\end{align}
Hence, noting  $k\leq x$, we obtain
\begin{align}
    k\sum_{n>N}\frac{(d(n)\log^2 k+2d_{(0,1)}(n)\log k+D_{(1)}(n)) e_k(-n\overline{h})}{2n}&\hspace{1mm}\frac{1}{2\pi i}\int_{-\delta-iT}^{-\delta+iT}\llap{=\quad\quad\hspace{8pt}}(\chi(s))^2\left(\frac{nx}{k^2}\right)^s s^{-1}ds\\
    &\ll k\hspace{1mm}x^{2\delta}.
\end{align}
Same estimates follow for the sums involving $(\chi_1'(s))^2$, $\chi_1(s)\chi_1'(s)$ and $(\chi_1(s))^2$ in \eqref{deltafinalform}, hence the tail $n>N$ of all the series can be omitted with an error $\ll k\hspace{0.4mm}x^{2\delta}$. Thus, the error term in \eqref{deltafinalform} is
$\ll x^{1+2\delta}T^{-1}+k\hspace{0.4mm}x^{2\delta}\ll
kx^{\frac{1}{2}+2\delta}N^{-\frac{1}{2}}$
as $N\ll x$. Then \eqref{deltafinalform} becomes
\begin{align}
    &\Delta(x,h/k)=\hspace{1mm}k\sum_{n\leq N}\frac{d(n)}{2n}e_k(-n\overline{h})\hspace{1mm}\frac{1}{2\pi i}\int_{-\delta-iT}^{-\delta+iT}\llap{=\quad\quad\hspace{8pt}}(\chi'(s))^2\left(\frac{nx}{k^2}\right)^s s^{-1}ds\notag\\
    &\hspace{1.7cm}-k\sum_{n\leq N}\frac{d(n)}{2n}e_k(-n\overline{h})\hspace{1mm}\frac{1}{2\pi i}\int_{-\delta-iT}^{-\delta+iT}\llap{=\quad\quad\hspace{8pt}}(\chi_1'(s))^2\left(\frac{nx}{k^2}\right)^s s^{-1}ds\notag\\
    &\hspace{1cm}-k\sum_{n\leq N}\frac{(d(n)\log k+d_{(0,1)}(n))}{n}e_k(-n\overline{h})\hspace{1mm}\frac{1}{2\pi i}\int_{-\delta-iT}^{-\delta+iT}\llap{=\quad\quad\hspace{8pt}}\chi(s)\chi'(s)\left(\frac{nx}{k^2}\right)^s s^{-1}ds\notag\\
    &\hspace{1cm}+k\sum_{n\leq N}\frac{(d(n)\log k+d_{(0,1)}(n))}{n}e_k(-n\overline{h})\hspace{1mm}\frac{1}{2\pi i}\int_{-\delta-iT}^{-\delta+iT}\llap{=\quad\quad\hspace{8pt}}\chi_1(s)\chi_1'(s)\left(\frac{nx}{k^2}\right)^s s^{-1}ds\notag\\
    &+k\sum_{n\leq N}\frac{(d(n)\log^2 k+2d_{(0,1)}(n)\log k+D_{(1)}(n))}{2n}e_k(-n\overline{h})\hspace{1mm}\frac{1}{2\pi i}\int_{-\delta-iT}^{-\delta+iT}\llap{=\quad\quad\hspace{8pt}}(\chi(s))^2\left(\frac{nx}{k^2}\right)^s s^{-1}ds\notag\\
    &-k\sum_{n\leq N}\frac{(d(n)\log^2 k+2d_{(0,1)}(n)\log k+D_{(1)}(n))}{2n}e_k(-n\overline{h})\hspace{1mm}\frac{1}{2\pi i}\int_{-\delta-iT}^{-\delta+iT}\llap{=\quad\quad\hspace{8pt}}(\chi_1(s))^2\left(\frac{nx}{k^2}\right)^s s^{-1}ds\notag\\
    &\hspace{11cm}+O(k\hspace{0.4mm}x^{\frac{1}{2}+2\delta}N^{-\frac{1}{2}}).\label{deltafinalform2}
\end{align}
Let $X_n=\frac{nx}{k^2}$. Define the following integrals :-
$$P_1(X_n):=\frac{1}{2\pi i}\int_{-\delta-iT}^{-\delta+iT}\llap{=\quad\quad\hspace{8pt}}(\chi'(s))^2\hspace{1mm}X_n^s s^{-1}ds,$$
$$P_2(X_n):=\frac{1}{2\pi i}\int_{-\delta-iT}^{-\delta+iT}\llap{=\quad\quad\hspace{8pt}}\chi(s)\chi'(s)\hspace{1mm}X_n^s s^{-1}ds,$$
and
$$P_3(X_n):=\frac{1}{2\pi i}\int_{-\delta-iT}^{-\delta+iT}\llap{=\quad\quad\hspace{8pt}}(\chi(s))^2\hspace{1mm}X_n^s s^{-1}ds.$$
Let us first evaluate $P_2(X_n)$. By applying integration by parts
\begin{align}
    P_2(X_n)&
    &=\frac{1}{4\pi i}\left[-\log X_n \int_{-\delta-iT}^{-\delta+iT}\llap{=\quad\quad\hspace{8pt}} (\chi(s))^2\frac{X_n^s}{s}ds+\int_{-\delta-iT}^{-\delta+iT}\llap{=\quad\quad\hspace{8pt}} (\chi(s))^2\frac{X_n^s}{s^2}ds\right]+O(X_n^{-\delta}T^{2\delta}),\label{P21}
\end{align}
Note that employing \eqref{chichi1}, the first integral in \eqref{P21} can be expressed as
\begin{align}
    \int_{-\delta-iT}^{-\delta+iT}\llap{=\quad\quad\hspace{8pt}} (\chi(s))^2 \log X_n\frac{X_n^s}{s}ds&=\frac{\log X_n}{2\pi^2}\left[\int_{-\delta-iT}^{-\delta+iT}\llap{=\quad\quad\hspace{8pt}}\right.\Gamma^2(1-s)(4\pi^2X_n)^s s^{-1}ds\notag\\
    &\left.-\int_{-\delta-iT}^{-\delta+iT}\llap{=\quad\quad\hspace{8pt}}\Gamma^2(1-s)\cos(\pi s)(4\pi^2X_n)^s s^{-1}ds\right].\label{P22}
\end{align}
Moreover, employing \eqref{gammasmallestimate} the second integral in \eqref{P21} can be estimated as
\begin{align}\label{P23}
    \int_{-\delta-iT}^{-\delta+iT}\llap{=\quad\quad\hspace{8pt}}\frac{X_n^s}{s^2}\chi^2(s)ds&
    \ll \left(\int_{-T}^{-3}+\int_{3}^{T}\right)\frac{X_n^{-\delta}|t|^{1+2\delta}}{|t|^2}dt\ll X_n^{-\delta}T^{2\delta},
\end{align}
Combining \eqref{P22} and \eqref{P23} in \eqref{P21},
\begin{align}
    P_2(X_n)=\frac{\log X_n}{4\pi^2} &\Bigg[\frac{1}{2\pi i}\int_{-\delta-iT}^{-\delta+iT}\llap{=\quad\quad\hspace{8pt}}\Gamma^2(1-s)\cos(\pi s)(4\pi^2X_n)^s s^{-1}ds\\
    &-\frac{1}{2\pi i}\left.\int_{-\delta-iT}^{-\delta+iT}\llap{=\quad\quad\hspace{8pt}}\Gamma^2(1-s)(4\pi^2X_n)^s s^{-1}ds\right]+O(X_n^{-\delta}T^{2\delta}).\label{P2final}
\end{align}
It is not difficult to obtain
\begin{align}\label{P3final}
  P_3(X_n)=\frac{1}{2\pi^2}& \Bigg[\frac{1}{2\pi i}\int_{-\delta-iT}^{-\delta+iT}\llap{=\quad\quad\hspace{8pt}}\Gamma^2(1-s)(4\pi^2X_n)^s s^{-1}ds\\
  &-\frac{1}{2\pi i}\int_{-\delta-iT}^{-\delta+iT}\llap{=\quad\quad\hspace{8pt}}\Gamma^2(1-s)\cos(\pi s)(4\pi^2X_n)^s s^{-1}ds\Bigg].
\end{align}
Next, we can observe that $P_1(X_n)$ can be rewritten as
\begin{align}\label{P11}
    P_1(X_n)=\frac{1}{2\pi i}\int_{-\delta-iT}^{-\delta+iT}\llap{=\quad\quad\hspace{8pt}}(\chi(s)\chi'(s))'X_n^s s^{-1}ds-\frac{1}{2\pi i}\int_{-\delta-iT}^{-\delta+iT}\llap{=\quad\quad\hspace{8pt}}\chi(s)\chi''(s)X_n^s s^{-1}ds.
\end{align}
By integration by parts, the first integral in \eqref{P11} can be expressed as
\begin{align}
    \int_{-\delta-iT}^{-\delta+iT}\llap{=\quad\quad\hspace{8pt}}(\chi(s)\chi'(s))'X_n^s s^{-1}ds=&-\log X_n\int_{-\delta-iT}^{-\delta+iT}\llap{=\quad\quad\hspace{8pt}}\chi(s)\chi'(s)X_n^s s^{-1}ds\notag\\
    &+\int_{-\delta-iT}^{-\delta+iT}\llap{=\quad\quad\hspace{8pt}}\chi(s)\chi'(s)X_n^s s^{-2}ds+O(X_n^{-\delta}T^{3\delta}).\label{P12}
\end{align}
We employ \eqref{chikderivative} at $\sigma=-\delta$ and $|t|=T$ to obtain
\begin{align}
    \chi(s)\chi'(s)=4(2\pi)^{2s-2}\Gamma^2(1-s)\sin^2(\pi s/2)\left(-\log\frac{T}{2\pi}\right)+O(T^{2\delta}),
\end{align}
and then substituting the above expression, we deduce
\begin{align}\label{P13}
    \int_{-\delta-iT}^{-\delta+iT}\llap{=\quad\quad\hspace{8pt}}\chi(s)\chi'(s)X_n^s s^{-2}ds
    &\ll \left(\int_{-T}^{-3}+\int_{3}^{T}\right)\left(\frac{|t|^{1+3\delta}}{|s|^2}+\frac{|t|^{2\delta}}{|s|^2}\right)X_n^{-\delta}dt\ll X_n^{-\delta}T^{3\delta},
\end{align}
using \eqref{gammasmallestimate}.
Then by \eqref{P12} and \eqref{P13} we get
\begin{align}\label{P14}
    \int_{-\delta-iT}^{-\delta+iT}\llap{=\quad\quad\hspace{8pt}}(\chi(s)\chi'(s))'X_n^s s^{-1}ds=-2\pi i\log X_n\hspace{1mm}P_2(X_n)+O(X_n^{-\delta}T^{3\delta}).
\end{align}
Next, we evaluate the second integral in \eqref{P11}. Since,
\begin{align}
    \chi(s)\chi''(s)=\left(\frac{|t|}{2\pi}\right)^{1+2\delta}e^{iE(t)}\left(-\log\frac{|t|}{2\pi}\right)^{2}+O(|t|^{3\delta}),
\end{align}
by \eqref{chi} and \eqref{chikderivative},
hence
\begin{align}\label{P15}
    \int_{-\delta-iT}^{-\delta+iT}\llap{=\quad\quad\hspace{8pt}}\chi(s)\chi''(s)X_n^s s^{-1}ds=\int_{-T}^{T}\llap{=\quad\hspace{0.1pt}}\left(\frac{|t|}{2\pi}\right)^{1+2\delta}e^{iE(t)}\left(-\log\frac{|t|}{2\pi}\right)^2\frac{X_n^{-\delta+it}}{-\delta+it}idt+O(X_n^{-\delta}T^{3\delta}).
\end{align}
Since $\frac{d}{dt} \left(\frac{e^{iE(t)}}{2i}\right)=\left(-\log\frac{|t|}{2\pi}\right)e^{iE(t)},$ by integration by parts, we have
\begin{align}
  &  \int_{-T}^{T}\llap{=\quad\hspace{0.1pt}}\left(\frac{|t|}{2\pi}\right)^{1+2\delta}e^{iE(t)}\left(-\log\frac{|t|}{2\pi}\right)^2\frac{X_n^{-\delta+it}}{-\delta+it}idt\\
   & =\frac{1}{2i}\log X_n\int_{-T}^{T}\llap{=\quad\hspace{0.1pt}}\left(\frac{|t|}{2\pi}\right)^{1+2\delta}\left(-\log\frac{|t|}{2\pi}\right)
    e^{iE(t)}\frac{X_n^{-\delta+it}}{-\delta+it}dt+O((1+T^{3\delta})X_n^{-\delta}).
\end{align}
Again using integration by parts, we have
\begin{align}
  &  \int_{-T}^{T}\llap{=\quad\hspace{0.1pt}}\left(\frac{|t|}{2\pi}\right)^{1+2\delta}e^{iE(t)}\left(-\log\frac{|t|}{2\pi}\right)^2\frac{X_n^{-\delta+it}}{-\delta+it}idt\\
    &=\frac{1}{4}\log^2 X_n\int_{-T}^{T}\llap{=\quad\hspace{0.1pt}}\left(\frac{|t|}{2\pi}\right)^{1+2\delta}e^{iE(t)}(1+O(1/|t|))\frac{X_n^{-\delta+it}}{-\delta+it}idt+O((1+T^{3\delta})X_n^{-\delta}\log^2 X_n)\\
    &=\frac{1}{4}\log^2 X_n\int_{-\delta-iT}^{-\delta+iT}\llap{=\quad\quad\hspace{8pt}}\chi^2(s)\frac{X_n^s}{s}ds+O((1+T^{3\delta})X_n^{-\delta}\log^2 X_n)\\
    &=\frac{i\pi}{2}\log^2 X_n\hspace{1mm}P_3(X_n)+O((1+T^{3\delta})X_n^{-\delta}\log^2 X_n).
\end{align}
Hence \eqref{P15} gives
\begin{align}\label{P16}
    \int_{-\delta-iT}^{-\delta+iT}\llap{=\quad\quad\hspace{8pt}}\chi(s)\chi''(s)X_n^s s^{-1}ds=\frac{i\pi}{2}\log^2 X_n\hspace{1mm}P_3(X_n)+O((1+T^{3\delta})X_n^{-\delta}\log^2 X_n).
\end{align}
Using \eqref{P14} and \eqref{P16} in \eqref{P11}, we have
\begin{align}\label{P1final}
    P_1(X_n)=-\log X_n\hspace{1mm}P_2(X_n)-\frac{1}{4}\log^2X_n\hspace{1mm}P_3(X_n)+O(X_n^{-\delta}\log^2X_n\hspace{1mm}T^{3\delta}).
\end{align}
Similarly, using Lemma \eqref{chi1estimate}, the other integrals in \eqref{deltafinalform2} can be evaluated.
Then,
\begin{align}
    \Delta(x,&h/k)=-\frac{k}{4\pi^2}\sum_{n\leq N}\frac{d(n)}{2n}e_k(-n\overline{h})\hspace{1mm}\log^2 \left(\frac{nx}{k^2}\right)\underbrace{\frac{1}{2\pi i}\int_{-\delta-iT}^{-\delta+iT}\llap{=\quad\quad\hspace{8pt}}\Gamma^2(1-s)\cos(\pi s)(4\pi^2X_n)^ss^{-1}ds}_{\mathcal{T}(X_n)}\notag\\
    &\hspace{1.2cm}-\frac{k}{2\pi^2}\sum_{n\leq N}\frac{(d(n)\log k+d_{(0,1)}(n))}{n}e_k(-n\overline{h})\hspace{1mm}\log \left(\frac{nx}{k^2}\right)\hspace{1mm}\mathcal{T}(X_n)\notag\\
    &-\frac{k}{\pi^2}\sum_{n\leq N}\frac{(d(n)\log^2 k+2d_{(0,1)}(n)\log k+D_{(1)}(n))}{2n}e_k(-n\overline{h})\hspace{1mm}\mathcal{T}(X_n)
    +O(k\hspace{0.4mm}x^{\frac{1}{2}+3\delta}N^{-\frac{1}{2}}),\label{deltafinalform3}
\end{align}
since the contribution of error terms in \eqref{P2final} and \eqref{P1final} and that of other integrals with their respective sums is $\ll k\hspace{0.4mm}x^{3\delta}$, 
and $X_n=\frac{nx}{k^2}$. Note that $\mathcal{T}(X_n)$ can be rewritten as
\begin{align}
    \mathcal{T}(X_n)=\frac{1}{2\pi i}\int_{-\delta-iT}^{-\delta+iT}\Gamma^2(1-s)\cos(\pi s)(4\pi^2X_n)^ss^{-1}ds+O(n^{-\delta}x^{\delta}).
\end{align}
We now extend the path of integration in $\mathcal{T}(X_n)$ to the infinite broken line through the points $\delta-i\infty$, $\delta-iT$, $-\delta-iT$, $-\delta+iT$, $\delta+iT$ and $\delta+i\infty$. Then we estimate the consequent bounds for the horizontal and infinite segments when we replace $\mathcal{T}(X_n)$ in \eqref{deltafinalform3} by the new integral. First, by \eqref{Tparameter}, \eqref{cosestimate} and \eqref{gammasmallestimate},
\begin{align}
&\frac{k}{4\pi^2}\sum_{n\leq N}\frac{d(n)}{2n}e_k(-n\overline{h})\hspace{1mm}\log^2 \left(\frac{nx}{k^2}\right)\underbrace{\frac{1}{2\pi i}\int_{-\delta\pm iT}^{\delta\pm iT}\Gamma^2(1-s)\cos(\pi s)(4\pi^2X_n)^ss^{-1}ds}_{\mathcal{F}(X_n)}\\
    &\ll k\sum_{n\leq N}n^{-1+\delta/2}\log^2\left(\frac{nT^2}{N}\right)\int_{-\delta}^{\delta}\left(\frac{4\pi^2 nx}{k^2T^2}\right)^{\sigma}d\sigma\\
    &\ll k^{1-\delta}x^{\delta}\sum_{n\leq N}n^{-1+\delta/2}\int_{-\delta}^{\delta}\left(\frac{n}{N}\right)^{\sigma}d\sigma\\
   &\ll k^{1-\delta}x^{\delta}\sum_{n\leq N}n^{-1+\delta/2}\left(\frac{N}{n}\right)^{\delta}\ll k x^{2\delta}.
\end{align}
 Similarly follows for the other sums, that is
\begin{align}
    \frac{k}{2\pi^2}\sum_{n\leq N}\frac{(d(n)\log k+d_{(0,1)}(n))}{n}e_k(-n\overline{h})\log \left(\frac{nx}{k^2}\right)\mathcal{F}(X_n)
\ll k\hspace{0.4mm}x^{2\delta},
\end{align}
and
\begin{align}
    \frac{k}{\pi^2}\sum_{n\leq N}\frac{(d(n)\log^2 k+2d_{(0,1)}(n)\log k+D_{(1)}(n))}{2n} e_k(-n\overline{h})\mathcal{F}(X_n)
    \ll k\hspace{0.4mm}x^{2\delta}.
\end{align}
Next, by \eqref{Tparameter}, \eqref{cosestimate}, \eqref{gammabigestimate} and  Lemma \eqref{expintegrallemma}, and taking $F(t)$ as defined in \eqref{Fdefn},
\begin{align}
    & \frac{k}{4\pi^2}\sum_{n\leq N}\frac{d(n) e_k(-n\overline{h})}{2n}\hspace{1mm}\log^2 \left(\frac{nx}{k^2}\right)\underbrace{\frac{1}{2\pi i}\int_{\delta+iT}^{\delta+i\infty}\Gamma^2(1-s)\cos(\pi s)(4\pi^2X_n)^ss^{-1}ds}_{\mathcal{G}(X_n)}\\
     &\ll k^{1-2\delta}x^{\delta}\sum_{n\leq N}n^{-1+2\delta}\log^2\left(\frac{nx}{k^2}\right)\left|\int_T^{\infty}(-i+O(1/t))t^{-2\delta}e^{iF(t)}dt\right|\\
     &\hspace{1cm}+\hspace{1mm}k^{1-2\delta}x^{\delta}\sum_{n\leq N}n^{-1+2\delta}\log^2\left(\frac{nx}{k^2}\right)\left(\int_T^{\infty}t^{-2\delta-1}e^{iF(t)}dt\right)\\
     &\ll k^{1-3\delta}x^{2\delta}\sum_{n\leq N}n^{-1+2\delta}\left[\left|\int_T^{\infty}t^{-2\delta}e^{iF(t)}dt\right|+T^{-2\delta}\right]\\
     &\ll k^{1-3\delta}x^{2\delta}T^{-2\delta}\sum_{n\leq N}n^{-1+2\delta}\left[1+\left(\log\left(\frac{N+1/2}{n}\right)\right)^{-1}\right]\ll k\hspace{0.4mm}x^{2\delta}.
\end{align}
Similarly, we may show that 
\begin{align}
    \frac{k}{2\pi^2}\sum_{n\leq N}\frac{(d(n)\log k+d_{(0,1)}(n))}{n}e_k(-n\overline{h})\log \left(\frac{nx}{k^2}\right)\mathcal{G}(X_n)\ll k\hspace{0.4mm}x^{2\delta},
\end{align}
and
\begin{align}
    \frac{k}{\pi^2}\sum_{n\leq N}\frac{(d(n)\log^2 k+2d_{(0,1)}(n)\log k+D_{(1)}(n))}{2n}&e_k(-n\overline{h})\mathcal{G}(X_n)\ll k\hspace{0.4mm}x^{2\delta},
\end{align}
and similarly for the integrals over $[\delta-i\infty, \delta-iT]$. 
These estimations show that \eqref{deltafinalform3} remains valid even if $\mathcal{T}(X_n)$ are replaced by the modified integrals, which are of the type $I_1$ in \eqref{firstintegral} for $a=0$ and $X_n=nxk^{-2}$, and thus equal to
$$2\pi^2(nx)^{1/2}k^{-1}Y_1(4\pi\sqrt{nx}/k)$$
by \eqref{firstintegral}. This gives
\begin{align}
    \frac{1}{2\pi i}\int_{C_0}\Gamma^2(1-s)\cos(\pi s)\left(\frac{4\pi^2nx}{k^2}\right)^s s^{-1}ds=-\sqrt{2}\pi(nx)^{1/4}&k^{-1/2}\cos\left(\frac{4\pi\sqrt{nx}}{k}-\frac{\pi}{4}\right)\notag\\
    &+O(n^{-3/4}x^{-3/4}k^{3/2}),\label{newintegralC0}
\end{align}
where $C_0$ is the contour described in Lemma \eqref{integraltobessel} for $a=0$, $\sigma_1=\delta$ and $\sigma_2=-\delta$.\\
The assertion \eqref{errorformula} now follows when $\mathcal{T}(X_n)$ in \eqref{deltafinalform3} is replaced by the integral in \eqref{newintegralC0}.
\section{Proof of Theorem \ref{Delta0meansqtheorem}}\label{pfofmeansq1}
For $T\ll k$, we trivially have
\begin{align}\label{meansq2}
        \int_1^T|\Delta(x,h/k)|^2dx\ll k^2\log^6k \hspace{1mm} T,
    \end{align}
    by \eqref{B_0formula} and \eqref{F(0)bound}. It is then sufficient to prove the formula  for $k\leq T$,
    \begin{align}\label{meansq1}
\int_T^{2T}|\Delta(x,h/k)|^2dx=\mathcal{H}(T)+O(k^2T^{1+\varepsilon})+O(k^{3/2}T^{5/4+\varepsilon})
\end{align}
where $\mathcal{H}(T)$ is given by
\begin{align}
\mathcal{H}(T)&=
\frac{kx^{3/2}}{6\pi^2}\sum_{n=1}^{\infty}\frac{1}{n^{3/2}}\bigg[\frac{1}{16}d^2(n)\sum_{i=0}^{4}\left(-\frac{2}{3}\right)^{i}{}^4P_i(\log x)^{4-i}\notag\\
  & +\frac{1}{2}d(n)f_1(n)\sum_{i=0}^{3}\left(-\frac{2}{3}\right)^{i}{}^3P_i\hspace{1mm}(\log x)^{3-i} \notag \\
    &+\frac{1}{2}(d(n)f_2(n)+f_1^2(n))\sum_{i=0}^{2}\left(-\frac{2}{3}\right)^{i}{}^2P_i\hspace{1mm}(\log x)^{2-i} \notag\\
    &+2f_1(n)f_2(n)\sum_{i=0}^{1}\left(-\frac{2}{3}\right)^{i}{}^1P_i\hspace{1mm}(\log x)^{1-i}+f_2^2(n)\bigg]\bigg|_T^{2T}\label{meansq11},
    \end{align}
    since we can substitute $T$ by $X/2$, $X/2^2$, $\cdots$, and \eqref{meansquareestimate} then follows easily from \eqref{meansq2} and \eqref{meansq1}.
    To prove \eqref{meansq1}, let $T\leq x\leq 2T$ and choose $N=T$ in \eqref{errorformula}, which can be written as
    \begin{align}\label{meansq3}
    \Delta(x,h/k)=S(x,h/k)+O(kx^\varepsilon).
    \end{align}
    Squaring out $|S(x,h/k)|^2$, we obtain
    \begin{align}\label{meansq4}
\int_T^{2T}|S(x,h/k)|^2dx=S_0+O(k\left(|S_{1}|+|S_2|\right)),
    \end{align}
    where
    \begin{align}
    S_0=\frac{k}{4\pi^2}\sum_{n\leq T}n^{-3/2}\int_T^{2T}&x^{1/2}\bigg[\frac{1}{16}d^2(n)\log^4x+\frac{1}{2}d(n)f_1(n)\log^3x \notag \\
    &+\frac{1}{2}\left(d(n)f_2(n)+f_1^2(n)\right)\log^2x+2f_1(n)f_2(n)\log x+f_2^2(n)\bigg]dx, \label{S0}
    \end{align}
    \begin{align}
        S_1=\underset{\substack{n,m\leq T\\ m\neq n}}{\sum}(m&n)^{-3/4}\int_T^{2T}x^{1/2}e(2(\sqrt{m}-\sqrt{n})\sqrt{x}/k)\bigg[\frac{1}{16}d(n)d(m)\log^4x+\frac{1}{2}d(n)f_1(m)\log^3x\\
        &+\left(\frac{1}{2}d(n)f_2(m)+f_1(n)f_1(m)\right)\log^2x+2f_1(n)f_2(m)\log x+f_2(n)f_2(m)\bigg]dx,
    \end{align}
    and
    \begin{align}
        S_2=\sum_{m,n\leq T}(m&n)^{-3/4}\int_T^{2T}x^{1/2}e(2(\sqrt{m}+\sqrt{n})\sqrt{x}/k)\bigg[\frac{1}{16}d(n)d(m)\log^4x+\frac{1}{2}d(n)f_1(m)\log^3x\\
        &+\left(\frac{1}{2}d(n)f_2(m)+f_1(n)f_1(m)\right)\log^2x+2f_1(n)f_2(m)\log x+f_2(n)f_2(m)\bigg]dx,
    \end{align}
    $f_1(n)$ and $f_2(n)$ are as given in \eqref{f1f2exp}. To establish the necessary bounds for $S_1$ and $S_2$, we first deduce the following lemma.
    \begin{lem}\label{meansqlem1}
        Let $a(n)$, $b(n)$ be arithmetical functions $\ll n^{\varepsilon}$ for all $n \geq n_0$ and $\varepsilon>0$. Then for $0\leq i\leq 4$, $i\in\mathbb{Z}$ and $T\geq 1$, we have
        \begin{align}
            \underset{\substack{n,m\leq T\\ m\neq n}}{\sum}a(m)b(n)(mn)^{-3/4}\int_T^{2T}x^{1/2}e(2(\sqrt{m}\pm\sqrt{n})\sqrt{x}/k)\log^i x\hspace{0.5mm}dx\ll kT^{1+\varepsilon}.\label{meansqlemineq}
        \end{align}
    \end{lem}
    \begin{proof}
    Using Lemma \eqref{expintegrallemma},
        \begin{align}
           \int_T^{2T}x^{1/2}e(2(\sqrt{m}\pm\sqrt{n})\sqrt{x}/k)\log^i x\hspace{0.5mm}dx\ll kT|\sqrt{n}\pm\sqrt{m}|^{-1}\log^i x.
        \end{align}
        Thus, the expression in \eqref{meansqlemineq} is
        \begin{align}
            &\ll k\hspace{0.5mm}T^{1+\varepsilon}\underset{\substack{n,m\leq T\\ m< n}}{\sum}m^{-3/4}n^{-1/4}(n\pm m)^{-1}\notag\\
           &\ll k\hspace{0.5mm}T^{1+\varepsilon}\sum_{m\leq T}m^{-1}\ll k\hspace{0.5mm}T^{1+\varepsilon}.
        \end{align}
    \end{proof} 
    Next, simplifying the expression for $S_0$
  in \eqref{S0}, we obtain
\begin{align} \label{Po}
S_0 = \mathcal{H}(T),
\end{align}
where $\mathcal{H}(T)$ is defined in \eqref{meansq11}.
We now apply Lemma \eqref{meansqlem1} to obtain the bounds
\begin{align}
 S_1\ll kT^{1+\varepsilon}\  \ \mbox{and}  \ \  S_2\ll kT^{1+\varepsilon}.
\end{align}
Using these estimates, together with equations \eqref{Po}, \eqref{meansq3}, and \eqref{meansq4}, and applying the Cauchy–Schwarz inequality, we prove  \eqref{meansq1}.
\par Hence combining \eqref{meansq2} with \eqref{meansq1}, and taking into account $k\ll X^{\frac12-\varepsilon}$ we conclude the proof.
\section{Proof of Theorem \ref{Generalizationtheorem}}\label{pfofgeneralizationth}
    Let $\varepsilon>0$ be arbitrarily small. By a well-known summation formula [\cite{AI}, equation (A.14), p. 487], we have for any $c>1$,
    \begin{align}
        B_a(x,h/k)=\frac{1}{2\pi i}\int_{(c)}F(s,h/k)x^{s+a}(s(s+1)\cdots (s+a))^{-1}ds.\label{Baperron}
    \end{align}
   We take $a\geq 1$ and deform the path of integration to a broken contour 
 $C_a$ passing through $-a/2+\varepsilon-i\infty$, $-a/2+\varepsilon-iT$, $-a-1/2-iT$, $-a-1/2+iT$, $-a/2+\varepsilon+iT$ and $-a/2+\varepsilon+i\infty$. Since $a\geq1$, we may choose 
$\varepsilon>0$ small enough such that $-a/2+\varepsilon<0$. Noting that $F(s,h/k)$ has a singularity at $s=1$, we define an auxiliary function as done in \eqref{auxfunc}. Then these two functions have the same order of magnitude in the region $|t|\geq 1$, $-a/2+\varepsilon\leq \sigma\leq c$. To estimate the auxiliary function on the line $\sigma=-a/2+\varepsilon$, we apply the functional equation \eqref{altfunceq}.
  \par  Using the Phragm\'{e}n-Lindel\"{o}f principle [\cite{HR}, (33.4)], we obtain that the integrand is 
    $$\ll |t|^{-\varepsilon} \hspace{2mm}\text{for} \hspace{2mm}|t|\geq 1\hspace{2mm} \text{and}\hspace{1mm} -a/2+\varepsilon\leq \sigma\leq c.$$ Thus the integral vanishes as $|t|\rightarrow \infty$. Choose $T=\frac{C_0\sqrt{x}}{k}$ to be sufficiently large, where $C_0>0$ is some constant. Then by Cauchy's residue theorem,
\begin{align}
    B_a(x,h/k)&=\left(\underset{s=-a}{Res}+\cdots+\right.\left.\underset{s=1}{Res}\right)F(s,h/k)x^{s+a}(s(s+1)\cdots(s+a))^{-1}\notag \\
    &+\underbrace{\frac{1}{2\pi i}\int_{C_a}F(s,h/k)x^{s+a}(s(s+1)\cdots(s+a))^{-1}ds}_{\mathcal{S}_a(x)}. \label{B_acauchyth}
\end{align}
    Note that the residue at $s=1$ gives the initial terms in \eqref{B_aformula}, and the sum of other residues equals
\begin{align}\sum_{i=0}^{a}\frac{(-1)^i}{i!(a-i)!}F(-i,h/k)x^{a-i}.\end{align}
    We now express $F(s,h/k)$ in $\mathcal{S}_a(x)$ in \eqref{B_acauchyth} using the functional equation \eqref{funceqexp} to obtain
    \begin{align}
\mathcal{S}_a(x)=&kx^a\Bigg[\sum_{n=1}^{\infty}\frac{d(n)}{2n}\left(e_k(n\overline{h})+e_k(-n\overline{h})\right)\frac{1}{2\pi i}\int_{C_a}(\chi'(s))^2\left(\frac{nx}{k^2}\right)^s (s(s+1)\cdots(s+a))^{-1}ds\notag\\
    &+\sum_{n=1}^{\infty}\frac{d(n)}{2n}\left(e_k(n\overline{h})-e_k(-n\overline{h})\right)\frac{1}{2\pi i}\int_{C_a}(\chi_1'(s))^2\left(\frac{nx}{k^2}\right)^s (s(s+1)\cdots(s+a))^{-1}ds\notag\\
    &-\sum_{n=1}^{\infty}\frac{(d(n)\log k+d_{(0,1)}(n))}{n}\left(e_k(n\overline{h})+e_k(-n\overline{h})\right)\times\\
    &\hspace{4cm}\times\frac{1}{2\pi i}\int_{C_a}\chi(s)\chi'(s)\left(\frac{nx}{k^2}\right)^s (s(s+1)\cdots(s+a))^{-1}ds\notag\\
&-\sum_{n=1}^{\infty}\frac{(d(n)\log k+d_{(0,1)}(n))}{n}\left(e_k(n\overline{h})-e_k(-n\overline{h})\right)\times\\
    &\hspace{4cm}\times\frac{1}{2\pi i}\int_{C_a}\chi_1(s)\chi_1'(s)\left(\frac{nx}{k^2}\right)^s (s(s+1)\cdots(s+a))^{-1}ds\notag\\
&+\sum_{n=1}^{\infty}\frac{(d(n)\log^2 k+2d_{(0,1)}(n)\log k+D_{(1)}(n))}{2n}\left(e_k(n\overline{h})+e_k(-n\overline{h})\right)\times\\
&\hspace{4cm}\times\frac{1}{2\pi i}\int_{C_a}(\chi(s))^2\left(\frac{nx}{k^2}\right)^s(s(s+1)\cdots(s+a))^{-1} ds\notag\\
    &+\sum_{n=1}^{\infty}\frac{(d(n)\log^2 k+2d_{(0,1)}(n)\log k+D_{(1)}(n))}{2n}\left(e_k(n\overline{h})-e_k(-n\overline{h})\right)\times\\
&\hspace{4cm}\times\frac{1}{2\pi i}\int_{C_a}(\chi_1(s))^2\left(\frac{nx}{k^2}\right)^s(s(s+1)\cdots(s+a))^{-1} ds\Bigg].\label{S_aformula}
    \end{align}
Since $a\geq 1$ the above expression is valid in the region $\Re(s)\leq -a/2+\varepsilon<0$. 
\par For $X_n=\frac{nx}{k^2}$, let us consider
     \begin{align}
        F_1(X_n)&=\frac{1}{2\pi i}\int_{C_a}(\chi(s))^2X_n^s(s(s+1)\cdots(s+a))^{-1} ds\\
        &=\frac{1}{2\pi^2}\bigg[\underbrace{\frac{1}{2\pi i}\int_{C_a}\Gamma^2(1-s)(4\pi^2X_n)^{s}(s(s+1)\cdots(s+a))^{-1}ds}_{\mathcal{H}_1(X_n)}\\
        &\hspace{1cm}-\underbrace{\frac{1}{2\pi i}\int_{C_a}\Gamma^2(1-s)\cos(\pi s)(4\pi^2X_n)^{s}(s(s+1)\cdots(s+a))^{-1}ds}_{\mathcal{H}_2(X_n)}\bigg],\label{F11}
    \end{align}
    where the second equality above holds by \eqref{chichi1}. Next we observe that
    \begin{align}
        F_2(X_n)&=\frac{1}{2\pi i}\int_{C_a}\chi(s)\chi'(s)X_n^s (s(s+1)\cdots(s+a))^{-1}ds \\
     &=\frac{1}{4\pi i}\int_{C_a}(\chi^2(s)'X_n^s (s(s+1)\cdots(s+a))^{-1}ds.\label{F21}
    \end{align}
    Using integration by parts,
\begin{align}
    F_2(X_n)=&\frac{\log X_n}{4\pi^2}\bigg[\frac{1}{2\pi i}\int_{C_a}\Gamma^2(1-s)\cos(\pi s)(4\pi^2X_n)^{s}(s(s+1)\cdots(s+a))^{-1}ds\\
    &-\frac{1}{2\pi i}\int_{C_a}\Gamma^2(1-s)(4\pi^2X_n)^{s}(s(s+1)\cdots(s+a))^{-1}ds\bigg]+O(X_n^{-a/2+1/2+\theta}T^{-1-2\theta})\notag\\
    &=\frac{\log X_n}{4\pi^2}\bigg [\mathcal{H}_2(X_n)-\mathcal{H}_1(X_n)\bigg]+O(X_n^{-a/2+\varepsilon}T^{-2\varepsilon})\label{F22}
\end{align}
where the error terms in \eqref{F22} comes by estimating the integrals of the type
\begin{align}\label{Error type}
    \int_{T}^{\infty}X_n^{-a/2+\varepsilon}|t|^{-1-\epsilon}dt\ll X_n^{-a/2+\varepsilon}T^{-\varepsilon}.
\end{align}
It remains to investigate $F_3(X_n)$. Again a simple calculation shows 
    \begin{align}
        F_3(X_n)&=\frac{1}{2\pi i}\int_{C_a}(\chi'(s))^2X_n^s (s(s+1)\cdots(s+a))^{-1}ds\\
        &=\underbrace{\frac{1}{2\pi i}\int_{C_a}(\chi(s)\chi'(s))'X_n^s (s(s+1)\cdots(s+a))^{-1}ds}_{\mathcal{J}_1(X_n)}\\
        &-\underbrace{\frac{1}{2\pi i}\int_{C_a}\chi(s)\chi''(s)X_n^s (s(s+1)\cdots(s+a))^{-1}ds}_{\mathcal{J}_2(X_n)}.\label{F31}
    \end{align}
    Using integration by parts in the first integral on right hand side of \eqref{F31}, we have
    \begin{align}
       \mathcal{J}_1(X_n)
        &=-(\log X_n)F_2(X_n)+O(X_n^{-a/2+\varepsilon}T^{-\varepsilon}),\label{F32}
    \end{align}
    where error term in expression  \eqref{F32} arises from estimating integrals of the type 
    \begin{align}
        \int_T^{\infty}X_n^{-a/2+\varepsilon}|t|^{-1-2\varepsilon}\log|t|dt\ll X_n^{-a/2+\varepsilon}T^{-\varepsilon}.
    \end{align}
    Using Lemma \eqref{chiestimate},
    \begin{align}
\mathcal{J}_2(X_n)=\int_{C_a}\llap{=\hspace{9.3pt}}\left(\frac{|t|}{2\pi}\right)^{1-2\sigma}\left(-\log\frac{|t|}{2\pi}\right)^2e^{iE(t)}X_n^{s}(s(s+1)\cdots(s+a))^{-1}ds+O(X_n^{-a/2+\varepsilon}T^{-\varepsilon}),\label{F32.5}
    \end{align}
    taking the integral over $[-a-1/2-iT_0,-a-1/2+iT_0]$ also in the error term above, for sufficiently large fixed $T_0<T$. Also note that the integral in \eqref{F32.5} over the horizontal segments of $C_a$ is
    \begin{align}
        \ll\int_{-a-1/2}^{-a/2+\varepsilon}T^{-a-2\sigma}\log^2T\hspace{1mm}X_n^{\sigma}d\sigma\ll X_n^{-a/2+\varepsilon}T^{-\varepsilon}.\label{F32.6}
    \end{align}
    Next, over the interval $T\leq |t|\leq \infty$ and $\sigma=-a/2+\varepsilon$, $\frac{d}{dt}\left(\frac{e^{iE(t)}}{2i}\right)=\left(-\log \frac{|t|}{2\pi}\right)e^{iE(t)}$. Hence, using integration by parts, the integral on the right hand side of \eqref{F32.5} equals
    \begin{align}
        \frac{\log X_n}{2i}\int_{T}^{\infty}\left(\frac{|t|}{2\pi}\right)^{1+a-2\varepsilon}\left(-\log\frac{|t|}{2\pi}\right)e^{iE(t)\hspace{1mm}}X_n^{s}&(s(s+1)\cdots(s+a)^{-1}ds+O(X_n^{-a/2+\varepsilon}T^{-\varepsilon}).
    \end{align}
    Again using integration by parts on the above integral, we get
    \begin{align}
        &\left(\int_{-a/2+\varepsilon-i\infty}^{-a/2+\varepsilon-iT}+\int_{-a/2+\varepsilon+iT}^{-a/2+\varepsilon+i\infty}\right)\chi(s)\chi''(s)X_n^s(s(s+1)\cdots(s+a))^{-1}ds\\
        &=\frac{1}{4}\log^2 X_n \left(\int_{-\infty}^{-T}+\int_{T}^{\infty}\right)\left(\frac{|t|}{2\pi}\right)^{1+a-2\varepsilon}e^{iE(t)}(1+O(1/t))\hspace{1mm}X_n^{s}(s(s+1)\cdots(s+a))^{-1}ds\\
        &\hspace{7cm}+O(X_n^{-a/2+\varepsilon}T^{-\varepsilon}(1+\log X_n+\log^2 X_n)).\label{F33}
    \end{align}
    Similarly, we can obtain
    \begin{align}
        &\left(\int_{-a-1/2-iT}^{-a-1/2-iT_0}+\int_{-a-1/2+iT_0}^{-a-1/2+iT}\right)\chi(s)\chi''(s)X_n^s(s(s+1)\cdots(s+a))^{-1}ds\\
        &=\frac{1}{4}\log^2 X_n \left(\int_{-T}^{-T_0}+\int_{T_0}^{T}\right)\left(\frac{|t|}{2\pi}\right)^{2a+2}e^{iE(t)}(1+O(1/t))\hspace{1mm}X_n^{s}(s(s+1)\cdots(s+a))^{-1}ds\\
        &\hspace{6cm}+O(X_n^{-a-1/2}T^{a+1+\varepsilon}(1+\log X_n+\log^2 X_n)).\label{F34}
    \end{align}
    Combining \eqref{F32.6}, \eqref{F33} and \eqref{F34} together with \eqref{F32.5} gives
    \begin{align}
        \mathcal{J}_2(X_n)=\frac{1}{4}\log^2{X_n}F_1(X_n)+O(X_n^{-a/2+\varepsilon}T^{-\varepsilon}(1+\log X_n+\log^2 X_n)).\label{J2final}
    \end{align}
    By \eqref{F31}, \eqref{F32} and \eqref{J2final},
    \begin{align}
        F_3(X_n)=-\frac{1}{4}\log^2 X_nF_1(X_n)-&\log X_nF_2(X_n)+O(X_n^{-a/2+\varepsilon}T^{-\varepsilon}(1+\log X_n+\log^2 X_n)).
\end{align}
Next employing \eqref{F11} and \eqref{F22} in the above expression and simplifying,  we obtain
       \begin{align} 
   F_3(X_n) &=\frac{1}{8\pi^2} \log^2 X_n(\mathcal{H}_1(X_n)-\mathcal{H}_2(X_n))+O(X_n^{-a/2+\varepsilon}T^{-\varepsilon}(1+\log X_n+\log^2 X_n)). \label{99}
    \end{align}
    Similarly, we may evaluate the integrals involving $\chi_1^2(s)$, $\chi_1\chi_1'(s)$ and $(\chi_1'(s))^2$ in \eqref{S_aformula}. Now employing the above facts together with \eqref{F11}, \eqref{F22} and \eqref{99} in \eqref{S_aformula}, we get
    \begin{align}
        \mathcal{S}_a(x)&=\frac{kx^a}{4\pi^2}\Bigg[\sum_{n=1}^{\infty}\frac{d(n)}{4n}\left(e_k(n\overline{h})+e_k(-n\overline{h})\right)\log^2\left(\frac{nx}{k^2}\right)\left\{\mathcal{H}_1\left(\frac{nx}{k^2}\right)-\mathcal{H}_2\left(\frac{nx}{k^2}\right)\right\}\\
    &+\sum_{n=1}^{\infty}\frac{d(n)}{4n}\left(e_k(n\overline{h})-e_k(-n\overline{h})\right)\log^2\left(\frac{nx}{k^2}\right)\left\{\mathcal{H}_1\left(\frac{nx}{k^2}\right)+\mathcal{H}_2\left(\frac{nx}{k^2}\right)\right\}\notag\\
    &+\sum_{n=1}^{\infty}\frac{(d(n)\log k+d_{(0,1)}(n))}{n}\left(e_k(n\overline{h})+e_k(-n\overline{h})\right)\log\left(\frac{nx}{k^2}\right)\left\{\mathcal{H}_1\left(\frac{nx}{k^2}\right)-\mathcal{H}_2\left(\frac{nx}{k^2}\right)\right\}\notag\\
&+\sum_{n=1}^{\infty}\frac{(d(n)\log k+d_{(0,1)}(n))}{n}\left(e_k(n\overline{h})-e_k(-n\overline{h})\right)\log\left(\frac{nx}{k^2}\right)\left\{\mathcal{H}_1\left(\frac{nx}{k^2}\right)+\mathcal{H}_2\left(\frac{nx}{k^2}\right)\right\}\notag\\
&+\sum_{n=1}^{\infty}\frac{(d(n)\log^2 k+2d_{(0,1)}(n)\log k+D_{(1)}(n))}{n}\left(e_k(n\overline{h})+e_k(-n\overline{h})\right)\left\{\mathcal{H}_1\left(\frac{nx}{k^2}\right)-\mathcal{H}_2\left(\frac{nx}{k^2}\right)\right\}\notag\\
    &+\sum_{n=1}^{\infty}\frac{(d(n)\log^2 k+2d_{(0,1)}(n)\log k+D_{(1)}(n))}{n}\left(e_k(n\overline{h})-e_k(-n\overline{h})\right)\left\{\mathcal{H}_1\left(\frac{nx}{k^2}\right)+\mathcal{H}_2\left(\frac{nx}{k^2}\right)\right\}\Bigg]\\
    &+O(k^{a+1+\varepsilon}x^{a/2+\varepsilon}).\label{S_aformula2}
    \end{align}
     Note that $\mathcal{H}_1\left(\frac{nx}{k^2}\right)$ and $\mathcal{H}_2\left(\frac{nx}{k^2}\right)$ are of the type $I_1$ and $I_2$ in Lemma \eqref{integraltobessel}. Utilizing  \eqref{firstintegral}, \eqref{secondintegral} in \eqref{S_aformula2} and then substituting the resultant expression in \eqref{B_acauchyth}, we can conclude \eqref{delta_aformula} for $a\geq 1$.\\
\section{Proof of Theorem \ref{Deltaameansqformula}}\label{pfofmeansq2}
The proof is similar to that of Theorem \eqref{Delta0meansqtheorem}.\\
We first establish a bound for $F(-n,h/k)$ for $0\leq n\leq a$, $a\in\mathbb{Z}$. By \eqref{Fexp},
\begin{align}
    F(-n,h/k)=k^{2n}\sum_{\alpha,\beta=1}^ke_k(\alpha\beta h)\big[&\zeta'(-n,\alpha/k)\zeta'(-n,\beta/k)-2\log k\hspace{1mm}\zeta(-n,\alpha/k)\zeta'(-n,\beta/k)\\
    &+\log^2 k\hspace{1mm}\zeta(-n,\alpha/k)\zeta(-n,\beta/k)\big]. 
\end{align}
Using \eqref{e_kpartialsumbound}, \eqref{zetaexp}, \eqref{zeta'exp} and properties of the Gamma function, we obtain
\begin{align}
    F(-n,h/k)\ll k^{2n+1}\log^3 k
\end{align}
by employing partial summation. Thus, if $T\ll k^2$, then by \eqref{B_aformula} we have 
$$\Delta_a(x,h/k)\ll k^{2a+1}\log^3 k .$$
Hence, for $T\ll k^2$, we trivially have
\begin{align}
    \int_T^{2T}|\Delta_a(x,h/k)|^2dx\ll k^{4a+2}\log^6k\hspace{1mm}T.\label{deltaameansq1}
\end{align}
Next, for $k^2<T$, let $T\leq x\leq 2T$, then formula \eqref{delta_aformula} can be written as 
\begin{align}
    \Delta_a(x,h/k)=U(x,h/k)+O(k^{a+1+\varepsilon}\hspace{0.5mm}x^{a/2+\varepsilon}),\label{deltaameansq2}
\end{align}
where
\begin{align}
    U(x,h/k)=\pi^{-\frac{1}{2}}(k/2\pi)^{a+\frac{1}{2}}x^{\frac{a}{2}+\frac{1}{4}}&\sum_{n=1}^{\infty}\Bigg\{\frac{1}{4}d(n)\log^2 x+\left(\frac{1}{2}d(n)\log n+d_{(0,1)}(n)\right)\log x\\
        &\hspace{0.5cm}+\left(\frac{1}{4}d(n)\log^2 n+d_{(0,1)}(n)\log n+D_{(1)}(n)\right)\Bigg\}\times\notag\\
        &\times n^{-\left(\frac{a}{2}+\frac{3}{4}\right)} e_k(-n\overline{h})\cos(4\pi\sqrt{nx}/k-\pi/4-\pi a/2).
\end{align}
Squaring $|U(x,h/k)|^2$ and integrating term by term, we obtain
\begin{align}
    \int_T^{2T}|U(x,h/k)|^2dx=U_0+O(k^{2a+1}(|U_1|+|U_2|)),\label{deltaameansq3}
\end{align}
where
\begin{align}
    U_0=\frac{1}{2\pi}(k/2\pi)^{2a+1}&\sum_{n=1}^{\infty}n^{-(a+3/2)}\int_T^{2T}x^{a+1/2}\bigg[\frac{1}{16}d^2(n)\log^4x+\frac{1}{2}d(n)f_1(n)\log^3x\\
    &+\frac{1}{2}\left(d(n)f_2(n)+f_1^2(n)\right)\log^2x+2f_1(n)f_2(n)\log x+f_2^2(n)\bigg]dx,
\end{align}
\begin{align}
    U_1&=\sum_{m,n=1 \atop{m\neq n}}^{\infty}(mn)^{-\left(\frac{a}{2}+\frac{3}{4}\right)}\int_T^{2T}x^{a+1/2}e(2(\sqrt{m}-\sqrt{n})\sqrt{x}/k)\times\\
        &\hspace{3.5cm}\times\bigg\{\frac{1}{16}d(n)d(m)\log^4x+\frac{1}{2}d(n)f_1(m)\log^3x\\
        &\hspace{1cm}+\left(\frac{1}{2}d(n)f_2(m)+f_1(n)f_1(m)\right)\log^2x+2f_1(n)f_2(m)\log x+f_2(n)f_2(m)\bigg\}dx,
\end{align}
\begin{align}
    U_2&=\sum_{m,n=1}^{\infty}(mn)^{-\left(\frac{a}{2}+\frac{3}{4}\right)}\int_T^{2T}x^{a+1/2}e(2(\sqrt{m}+\sqrt{n})\sqrt{x}/k)\times\\
        &\hspace{3.5cm}\times\bigg\{\frac{1}{16}d(n)d(m)\log^4x+\frac{1}{2}d(n)f_1(m)\log^3x\\
        &\hspace{1cm}+\left(\frac{1}{2}d(n)f_2(m)+f_1(n)f_1(m)\right)\log^2x+2f_1(n)f_2(m)\log x+f_2(n)f_2(m)\bigg\}dx,
\end{align}
where $f_1(n)$ and $f_2(n)$ are same as defined in \eqref{f1f2exp}. We now use the following lemma to establish bounds on $U_1$ and $U_2$.
\begin{lem}\label{meansqlem2}
        Let $a(n)$, $b(n)$ be arithmetical functions $\ll n^{\varepsilon}$ for all $\varepsilon>0$. Then for $0\leq j\leq 4$, $j\in\mathbb{Z}$ and $T\geq 1$, we have
        \begin{align}
            \sum_{m,n=1 \atop{m\neq n}}^{\infty}a(m)b(n)(mn)^{-\left(\frac{a}{2}+\frac{3}{4}\right)}\int_T^{2T}x^{a+1/2}e(2(\sqrt{m}\pm \sqrt{n})\sqrt{x}/k)\log^jx\ll k\hspace{0.5mm}T^{a+1+\varepsilon}.\label{meansqlemineq2}
        \end{align}
    \end{lem}
    \begin{proof}
    Using Lemma \eqref{expintegrallemma},
        \begin{align}
           \int_T^{2T}x^{a+1/2}e(2(\sqrt{m}\pm\sqrt{n})\sqrt{x}/k)\log^j x\hspace{0.5mm}dx\ll kT^{a+1}|\sqrt{n}\pm\sqrt{m}|^{-1}\log^j x.\label{meansqineq2}
        \end{align}
        Thus, the expression in \eqref{meansqlemineq2} is
        \begin{align}
            &\ll k\hspace{0.5mm}T^{a+1+\varepsilon}\sum_{m,n=1 \atop{m<n}}^{\infty}m^{-\left(\frac{a}{2}+\frac{3}{4}\right)}n^{-\left(\frac{a}{2}+\frac{1}{4}\right)}(n\pm m)^{-1}\ll k\hspace{0.5mm}T^{a+1+\varepsilon}\sum_{m=1}^{\infty}m^{-(a+1)}\ll k\hspace{0.5mm}T^{a+1+\varepsilon}.
        \end{align}
    \end{proof}
Using Lemma \eqref{meansqlem2}, we have $U_1\ll kT^{a+1+\varepsilon}$ and $U_2\ll kT^{a+1+\varepsilon}$, and integrating the sum $U_0$ term by term using integration by parts, \eqref{deltaameansq3} gives
\begin{align}
&\int_T^{2T}|U(x,h/k)|^2dx=\frac{1}{2\pi (a+3/2)}(k/2\pi)^{2a+1}x^{a+3/2}\sum_{n=1}^{\infty}n^{-(a+3/2)}\times\\
&\times\bigg[\frac{1}{16}d^2(n)\sum_{i=0}^{4}\left(\frac{-1}{a+3/2}\right)^{i}\frac{4!}{(4-i)!}\hspace{0.5mm}\log^{4-i}x+\frac{1}{2}d(n)f_1(n)\sum_{i=0}^{3}\left(\frac{-1}{a+3/2}\right)^{i}\frac{3!}{(3-i)!}\hspace{0.5mm}\log^{3-i}x\\
&\hspace{1cm}+\frac{1}{2}\left(d(n)f_2(n)+f_1^2(n)\right)\sum_{i=0}^{2}\left(\frac{-1}{a+3/2}\right)^{i}\frac{2!}{(2-i)!}\hspace{0.5mm}\log^{2-i}x\\
&\hspace{1cm}+2f_1(n)f_2(n)\sum_{i=0}^{1}\left(\frac{-1}{a+3/2}\right)^{i}\frac{1!}{(1-i)!}\hspace{0.5mm}\log^{1-i}x+f_2^{2}(n)\bigg]\Bigg|_{T}^{2T}+O(k^{2a+2}T^{a+1+\varepsilon}).
\end{align}
This along with \eqref{deltaameansq2} and Cauchy's inequality gives
\begin{align}
    \int_T^{2T}|\Delta_a(x,h/k)|^2dx=k^{2a+1}p(x,n)\bigg|_T^{2T}+O(k^{2a+2}T^{a+1+\varepsilon})+O(k^{2a+\frac{3}{2}}T^{a+\frac{5}{4}+\varepsilon}),\label{deltaameansq4}
\end{align}
where
\begin{align}
    p(x,n)=\frac{1}{a+3/2}(2\pi)^{-2a-2}x^{a+3/2}&\sum_{n=1}^{\infty}n^{-(a+3/2)}\bigg[\frac{1}{16}d^2(n)\sum_{i=0}^{4}\left(\frac{-1}{a+3/2}\right)^{i}\frac{4!}{(4-i)!}\hspace{0.5mm}\log^{4-i}x\\
    &\hspace{1.3cm}+\frac{1}{2}d(n)f_1(n)\sum_{i=0}^{3}\left(\frac{-1}{a+3/2}\right)^{i}\frac{3!}{(3-i)!}\hspace{0.5mm}\log^{3-i}x\\
&+\frac{1}{2}\left(d(n)f_2(n)+f_1^2(n)\right)\sum_{i=0}^{2}\left(\frac{-1}{a+3/2}\right)^{i}\frac{2!}{(2-i)!}\hspace{0.5mm}\log^{2-i}x\\
&+2f_1(n)f_2(n)\sum_{i=0}^{1}\left(\frac{-1}{a+3/2}\right)^{i}\frac{1!}{(1-i)!}\hspace{0.5mm}\log^{1-i}x+f_2^{2}(n)\bigg],
\end{align}
for $k^2< T$. Hence combining \eqref{deltaameansq1} and \eqref{deltaameansq4} and noting $k\ll X^{\frac12-\varepsilon}$,  we deduce \eqref{deltaameansqestimate}.
\section{Acknowledgement}
The second author's research is funded by the SERB under File No. MTR/2023/000837 and CRG/2023/002698.

\end{document}